\documentclass[a4paper, 11pt]{amsart}

\pdfoutput=1

\usepackage[T1]{fontenc}
\usepackage{mathptmx,amssymb,amscd,latexsym,eulervm}
\usepackage{amsmath,amsthm,mathrsfs,bm,mathtools}
\usepackage[mathscr]{eucal}
\usepackage{mathdots}

\usepackage{array}
\usepackage{needspace}
\usepackage{graphicx}
\usepackage{xcolor}
\usepackage{soul}
\usepackage{tikz}
\usetikzlibrary{arrows,shapes,positioning,calc,patterns,cd,matrix}

\usepackage{setspace}
\usepackage[inner=2.4cm,outer=2.4cm,top=3cm,bottom=2.7cm,marginparwidth=2cm]{geometry}
\usepackage{microtype}
\usepackage{aliascnt}
\usepackage[pagebackref,colorlinks=true,linkcolor=blue,citecolor=blue,urlcolor=blue]{hyperref}
\hypersetup{
  pdftitle={Transcendental Epsilon Multiplicity via Divisor Volumes},
  pdfauthor={Sudipta Das, Stephen Landsittel, Vinh Anh Pham},
  pdfsubject={Epsilon multiplicity, divisor volumes, and transcendence},
  pdfkeywords={epsilon multiplicity, divisor volume, section ring, projective bundle, Baker's theorem}
}
\usepackage[capitalize,nameinlink]{cleveref}

\allowdisplaybreaks
\definecolor{myorange}{RGB}{255,160,70}
\definecolor{mygreen}{RGB}{0,128,0}
\definecolor{darkgreen}{RGB}{0,166,0}

\newtheorem{theorem}{Theorem}[section]
\theoremstyle{definition}

\theoremstyle{plain}
\newtheorem{headthm}{Theorem}

\newaliascnt{headcor}{headthm}

\aliascntresetthe{headcor}

\newaliascnt{headconj}{headthm}

\aliascntresetthe{headconj}

\newaliascnt{corollary}{theorem}
\newtheorem{corollary}[corollary]{Corollary}
\aliascntresetthe{corollary}

\newaliascnt{claim}{theorem}

\aliascntresetthe{claim}

\newaliascnt{lemma}{theorem}
\newtheorem{lemma}[lemma]{Lemma}
\aliascntresetthe{lemma}

\newaliascnt{conjecture}{theorem}

\aliascntresetthe{conjecture}

\newaliascnt{proposition}{theorem}
\newtheorem{proposition}[proposition]{Proposition}
\aliascntresetthe{proposition}

\theoremstyle{definition}
\newaliascnt{definition}{theorem}

\aliascntresetthe{definition}

\newaliascnt{notation}{theorem}

\aliascntresetthe{notation}

\newaliascnt{example}{theorem}

\aliascntresetthe{example}

\newaliascnt{examples}{theorem}

\aliascntresetthe{examples}

\newaliascnt{remark}{theorem}
\newtheorem{remark}[remark]{Remark}
\aliascntresetthe{remark}

\newaliascnt{question}{theorem}
\newtheorem{question}[question]{Question}
\aliascntresetthe{question}

\numberwithin{equation}{section}
\numberwithin{figure}{section}

\DeclareMathOperator{\Proj}{Proj}

\DeclareMathOperator{\conv}{conv}

\DeclareMathOperator{\Sym}{Sym}

\DeclareMathOperator{\vol}{vol}

\newcommand{\m}{\mathfrak m}

\newcommand{\OO}{\mathcal O}

\newcommand{\F}{\mathcal F}

\newcommand{\Mcal}{\mathcal M}

\newcommand{\CC}{\mathbb C}

\newcommand{\NN}{\mathbb Z_{\ge 0}}

\newcommand{\QQ}{\mathbb Q}

\newcommand{\RR}{\mathbb R}

\newcommand{\ZZ}{\mathbb Z}

\newcommand{\eps}{\varepsilon}

\renewcommand{\leq}{\leqslant}
\renewcommand{\geq}{\geqslant}

\begin{document}

\title[Transcendental Epsilon Multiplicity via Divisor Volumes]{Transcendental Epsilon Multiplicity via Divisor Volumes}

\author[Sudipta Das]{Sudipta Das}
\address{Sudipta Das\\ School of Mathematics, Tata Institute of Fundamental Research, Dr. Homi Bhabha Road, Colaba, Mumbai 400005, Maharashtra, India}
\email{sudiptad@math.tifr.res.in}

\author[Stephen Landsittel]{Stephen Landsittel}
\address{Stephen Landsittel\\ Institute of Mathematics, Hebrew University, Givat Ram, Jerusalem 91904, Israel}
\email{stephen.landsittel@mail.huji.ac.il}

\author[Vinh Anh Ph{\d{A}}m]{Vinh Anh Ph{\d{A}}m}
\address{Vinh Anh Ph{\d{A}}m\\ Department of Mathematics, Tulane University, 6823 St. Charles Avenue, New Orleans, LA 70118, USA}
\email{vpham1@tulane.edu}


\begin{abstract}
We prove that epsilon multiplicity can take transcendental values.  The
main structural result is a one-ideal formula for section rings:
under natural positivity hypotheses, the epsilon multiplicity of an ideal
generated in one degree is equal to an integral of a divisor-volume function.
This formula transports an asymptotic colength invariant of ideals to the
geometry and arithmetic of divisor volumes.  To produce a transcendental
value, we combine the formula with a shifted projective-bundle construction
inspired by Borntr\"ager and Nickel.  The shift places the construction in the
positivity range required by the one-ideal formula while preserving the
underlying disk geometry of the volume computation.  Reversing the order of
integration reduces the resulting integral to three integrals of rational
functions.  Their arctangent terms cancel exactly, whereas the remaining real
logarithms form an explicit algebraic linear combination whose value is
positive.  Baker's theorem then implies transcendence.  Consequently, there
exists a homogeneous ideal in a normal standard graded domain whose epsilon
multiplicity is transcendental.
\end{abstract}

\maketitle

\section{Introduction}

Hilbert--Samuel multiplicity is one of the basic numerical invariants of
local algebra.  For an $\mathfrak m$-primary ideal in a Noetherian local ring
$(R,\mathfrak m)$ it is governed by a Hilbert function which is eventually polynomial and controls the singularities of the underlying ring.  Epsilon
multiplicity, introduced by Ulrich and Validashti \cite{UlrichValidashti}, measures ideals beyond the $\mathfrak m$-primary case.  It is
defined by the asymptotic growth of the local cohomology lengths
$$
        \lambda_R\bigl(H^0_{\mathfrak m}(R/I^n)\bigr)
        =
        \lambda_R\bigl((I^n:\mathfrak m^\infty)/I^n\bigr),
$$
and it has become a useful invariant in the study of integral dependence,
local volumes, and asymptotic algebraic geometry; see, for example,
\cite{UlrichValidashti,CutLan24,Fulger}.

The arithmetic behavior of epsilon multiplicity is much less rigid than that
of Hilbert--Samuel multiplicity.  Cutkosky, H\`a, Srinivasan, and Theodorescu
proved that epsilon multiplicity can be irrational \cite{CHST}.  Cutkosky's
work and its refinements relate epsilon multiplicity to volume-theoretic
limits in several settings \cite{Cut11,CutLan24,Lan25}, and recent density
function methods reveal further structure behind these limits \cite{DRT25}.
These results suggest that epsilon multiplicity is not merely a generalized
multiplicity, but an invariant whose values can record the geometry and
arithmetic of asymptotic section counts.  A recent preprint of Sarkar
\cite{Sarkar26} proves a multiplicity--volume formula for epsilon multiplicity
through limits attached to families of ideals.  The formula established here
is of a different nature: for a single homogeneous ideal in a section ring, it
identifies epsilon multiplicity directly with an integral of a divisor-volume
function.

The first main result of this paper shows that the arithmetic complexity of epsilon multiplicity goes
beyond irrationality.

\begin{headthm}\label{thm:intro-main}
There exists a $5$-dimensional normal standard graded domain $S$ and a
homogeneous ideal $J\subseteq S$ such that
$$
        \varepsilon(JS_{S_+})
$$
is transcendental.
\end{headthm}

To the best of our knowledge, this is the first known existence result of a transcendental
epsilon multiplicity.  The construction does not insert a transcendental
constant into the equations.  The transcendental number is produced by a
volume function of a divisor, and the point is to show that this arithmetic
information survives after passing back to epsilon multiplicity in a local ring.

The structural input is a local-to-global formula for a single ideal.  Let $X$ be a
normal projective variety, let $L$ be an ample line bundle with standard
graded section ring $R(X,L)$, and let $D$ be an effective Cartier divisor.  We
view $H^0(X,cL-D)$ as a subspace of the degree-$c$ piece
$R(X,L)_c=H^0(X,cL)$ and set
$$
        J=R(X,L)\cdot H^0(X,cL-D).
$$
Under the positivity hypotheses stated in Theorem~\ref{thm:one-ideal}, this
single ideal has epsilon multiplicity given by a divisor-volume integral.

\begin{headthm}\label{thm:intro-one-ideal}
In the setting of Theorem~\ref{thm:one-ideal}, one has
$$
        \varepsilon(JR_{R_+})
        =
        (r+1)\int_\tau^c \vol_X(tL-D)\,dt,
$$
where $r=\dim X$ and
$$
        \tau=\inf\{t\geq 0\mid tL-D\text{ is pseudo-effective}\}.
$$
\end{headthm}

Thus a numerical invariant defined by saturation of powers of a single ideal is
transported to the geometry of divisor volumes.  This is the mechanism behind
Theorem~\ref{thm:intro-main}.  It provides a bridge
$$
\text{multiplicity theory}
\longrightarrow
\text{section rings}
\longrightarrow
\text{divisor volumes}
\longrightarrow
\text{transcendence theory}.
$$
For the purposes of commutative algebra, the formula explains how epsilon
multiplicity can retain information that is invisible to Hilbert polynomials.
For the purposes of algebraic geometry, it places epsilon multiplicity among
asymptotic invariants controlled by positivity and volume functions.

It remains to produce a divisor-volume integral with a transcendental value
while keeping the positivity needed for Theorem~\ref{thm:one-ideal}.  We do this with a shifted projective-bundle construction inspired by
Borntr\"ager and Nickel \cite{BN}.  Their work supplies the circular numerical
model and the split projective-bundle strategy.  The particular triangle used
here is fixed explicitly in Section~\ref{sec:BN}, and every volume formula
needed in the proof is derived independently.  The shift moves the relevant
line bundles into the ample range required by the one-ideal formula, without
changing the underlying triangle that controls the disk integral.  The volume
integral is reduced to a weighted area integral over the intersection of a
fixed triangle with a moving disk.  Reversing the order of integration and
then rationalizing the three polar boundary integrals produces an explicit
algebraic linear combination of logarithms of positive algebraic numbers
whose value is positive.  Baker's theorem then gives transcendence.

A short filtration example in Section~\ref{sec:filtration-example} serves only
as a model: it shows that it is not difficult to find filtrations whose multiplicity is transcendental, while the main issue is to obtain such behavior from powers of a
single ideal.  The one-ideal formula and the shifted projective-bundle
construction solve precisely that problem.

The paper is organized as follows.  Section~\ref{prelim} collects the
necessary material on saturation, section rings, volumes, and Baker's theorem.
Section~\ref{sec:from-filtrations-to-one-ideal} proves the one-ideal
integral formula after the filtration model.  Section~\ref{sec:BN} constructs
the shifted projective-bundle family and verifies the positivity hypotheses.
Section~\ref{sec:arithmetic-input} evaluates the resulting divisor-volume
integral and isolates the transcendental logarithmic contribution.
Section~\ref{sec:main} proves Theorem~\ref{thm:intro-main}.  The final section
records a few questions about the arithmetic range of epsilon multiplicity.

\section{Preliminaries}\label{prelim}

This section collects the algebraic, geometric, and arithmetic inputs used in the proof of the
one-ideal formula and in the final argument for transcendence of epsilon multiplicity. We have written the material in a
slightly expansive form, because the proof later passes between three languages: local algebra, divisorial section rings, and divisor-volume integrals.

Throughout this paper, $k$ denotes a field.  Varieties are assumed to be integral and projective over $k$
unless explicitly stated otherwise.  We write divisors and tensor powers of line bundles additively.
Thus, if $L$ is a line bundle on $X$, then $mL$ refers to the sheaf $L^{\otimes m}$, and if $D$ is a Cartier divisor then an expression such as $mL-D$ is defined as the sheaf
$$
L^{\otimes m}\otimes \OO_X(-D).
$$
With this convention,
$$
R(X,L):=\bigoplus_{m\ge 0}H^0(X,mL)
$$
denotes the section ring of $L$.  When $D$ is a Cartier divisor, we often write $H^0(X,D)$ for
$H^0(X,\OO_X(D))$.

\subsection{Epsilon multiplicity and saturation}

Let $(R,\m)$ be a Noetherian local ring of dimension $d$, and let $I\subseteq R$ be an ideal.
The epsilon multiplicity of $I$ is defined as the following limit superior
$$
\eps(I)
=
d!\limsup_{n\to\infty}
\frac{\lambda_R\bigl(H^0_{\m}(R/I^n)\bigr)}{n^d}.
$$
Here
$$
H^0_{\m}(R/I^n)
=
\{\, x+I^n\in R/I^n\mid \m^i(x+I^n)=0
\text{ for some }i\geq 0\,\}
$$
is the $\m$-torsion submodule of $R/I^n$.  Equivalently,
$$
H^0_{\m}(R/I^n)\cong (I^n:\m^\infty)/I^n.
$$
In particular
$$
\eps(I)
=
d!\limsup_{n\to\infty}
\frac{\lambda_R\bigl((I^n:\m^\infty)/I^n\bigr)}{n^d}.
$$
In the setting used below, the limit superior is a limit by Theorem~\ref{thm:one-ideal}.

We shall use this invariant in a graded setting.  Let
$$
S=\bigoplus_{m\ge 0}S_m
$$
be a standard graded domain over $k$, and set
$$
S_+:=\bigoplus_{m>0}S_m.
$$
If $J\subseteq S$ is a homogeneous ideal, its saturation with respect to $S_+$ is
$$
J^{\rm sat}:=J:S_+^\infty=\bigcup_{i\ge 1}(J:S_+^i).
$$
With this notation we may rewrite zero-th local cohomology as follows
$$
H^0_{S_+}(S/J)\cong J^{\rm sat}/J.
$$
If $N$ is a finite-length graded $S$-module, then localizing at $S_+$ does not change its length:
$$
\lambda_S(N)=\lambda_{S_{S_+}}(N_{S_+}).
$$
Consequently, if $d=\dim S$, then
$$
\eps(JS_{S_+})
=
d!\limsup_{n\to\infty}
\frac{\lambda_S\bigl((J^n)^{\rm sat}/J^n\bigr)}{n^d}.
$$
This is the form of epsilon multiplicity that will be used in the one-ideal formula.

More generally one defines the epsilon multiplicity of a filtration of ideals.  A filtration of ideals in a local
ring $(R,\m)$ is a family of ideals $\{J_n\}_{n\ge 0}$ such that $J_0=R$, which is graded, in the sense that $J_mJ_n\subseteq J_{m+n}$
for all $m$ and $n$, and decreases: $J_1\supseteq J_2\supseteq J_3\supseteq\cdots$. Cutkosky and Sarkar define the epsilon multiplicity of a filtration $\{J_n\}$ in \cite{CutSarkar} as the following limit superior
$$
\eps(\{J_n\})
=
d!\limsup_{n\to\infty}
\frac{\lambda_R\bigl(H^0_{\m}(R/J_n)\bigr)}{n^d}.
$$
The filtration example in Section~\ref{sec:filtration-example} should be viewed as a model:
transcendence is easy to see at the level of filtrations, while the main work of the paper
is to realize such arithmetic behavior from a single ideal.

\subsection{Section rings and divisorial ideals}\label{subsec:section-rings-divisorial-ideals}

Let $X$ be a normal projective variety, and let $L$ be an ample line bundle such that the section
ring
$$
S:=R(X,L)=\bigoplus_{m\ge 0}H^0(X,mL)
$$
is standard graded.  Then $L$ is globally generated and
$$
X\cong \Proj S.
$$

Indeed note that $S=R(X,L)$ is generated in degree one, so the linear system $|L|$ has no
base points: otherwise every section of every $mL$ would vanish at the same
point, contradicting ampleness. Hence $L$ is globally generated. The canonical
morphism
$$
X \longrightarrow \Proj R(X,L)
$$
associated to an ample invertible sheaf is an open immersion with dense image;
see \cite[Tag~01Q1]{Stacks}. Since $S$ is generated by $S_1$ and $L$ is
globally generated, the standard opens $D_+(s)$, $s\in S_1$, are all in the image.
Thus the morphism is surjective, and hence $X\cong \Proj S$.

For a coherent sheaf $\F$ on $X$, put
$$
\Gamma_*(\F)
:=
\bigoplus_{m\ge 0}H^0(X,\F\otimes L^{\otimes m}).
$$

Let $D$ be an effective Cartier divisor on $X$.  We define the corresponding divisorial ideal in
the section ring by
$$
I_D
:=
\Gamma_*(\OO_X(-D))
=
\bigoplus_{m\ge 0}H^0(X,mL-D)
\subseteq R(X,L).
$$
More generally,
$$
I_{nD}
:=
\Gamma_*(\OO_X(-nD))
=
\bigoplus_{m\ge 0}H^0(X,mL-nD).
$$
Since $D$ is Cartier, $\OO_X(-D)$ is invertible, and hence
$$
\widetilde{I_D}\cong \OO_X(-D),
\qquad
\widetilde{I_{nD}}\cong \OO_X(-nD).
$$
Moreover, sheafification commutes with tensor products of invertible sheaves on standard opens,
so that
$$
\widetilde{I_D^n}\cong \OO_X(-nD).
$$
Thus the product $I_D^n$ and the divisorial module $I_{nD}$ define the same coherent ideal sheaf
on $X$.

The difference between these two graded ideals is precisely a saturation phenomenon.  If
$I\subseteq S$ is a homogeneous ideal, then, since $S$ is standard graded, the
homogeneous ideal corresponding to $\widetilde I$ is the saturation
$I:S_+^\infty$; see \cite[Tag~084M]{Stacks}.  Equivalently, in the present
section-ring setting,
$$
\Gamma_*(\widetilde I)=I:S_+^\infty.
$$
Applying this to the homogeneous ideal $I=I_D^n$ gives
$$
I_D^n:S_+^\infty
=
\Gamma_*(\widetilde{I_D^n})
=
\Gamma_*(\OO_X(-nD))
=
I_{nD}.
$$
This observation is the basic bridge between the graded local cohomology in epsilon
multiplicity and the geometry of divisors on $X$.

\subsection{Global generation, regularity, and multiplication of sections}

We next isolate the positivity hypotheses used to compare powers of one degree-generated ideal
with the corresponding divisorial modules.

Recall first that a line bundle $\Mcal$ on a projective variety $X$ is globally generated if the
evaluation map
$$
H^0(X,\Mcal)\otimes_k \OO_X
\longrightarrow
\Mcal
$$
is surjective.  Equivalently, the complete linear system $|\Mcal|$ has no base points.  We emphasize
the evaluation-map formulation because it is exactly what is used below for the line bundle
$L^{\otimes c}\otimes \OO_X(-D)$.

Let $L$ be a globally generated ample line bundle on $X$.  A coherent sheaf $\F$ is said to be
$0$-regular with respect to $L$ if
$$
H^i(X,\F\otimes L^{-i})=0
\qquad\text{for all }i>0.
$$
We shall use the following standard consequences of Castelnuovo--Mumford regularity; see
Lazarsfeld \cite[Theorem~1.8.5]{Lazarsfeld}.

\begin{theorem}\label{thm:regularity-package}
Let $L$ be a globally generated ample line bundle on a projective variety $X$, and let $\F$ be a
coherent sheaf on $X$.
\begin{enumerate}
\item If $\F$ is $0$-regular with respect to $L$, then $\F\otimes L^m$ is globally generated for
every $m\ge 0$.
\item If $\F$ is $0$-regular with respect to $L$, then the multiplication maps
$$
H^0(X,\F\otimes L^m)\otimes H^0(X,L)
\longrightarrow
H^0(X,\F\otimes L^{m+1})
$$
are surjective for all $m\ge 0$.
\item By iterating the multiplication maps in $(2)$, if $\F$ is $0$-regular with respect to $L$, then the graded module
$$
\Gamma_*(\F)
=
\bigoplus_{m\ge 0}H^0(X,\F\otimes L^{\otimes m})
$$
is generated in degree $0$.  In particular, for every $m\ge 0$, the natural multiplication map
$$
H^0(X,\F)\otimes H^0(X,mL)
\longrightarrow
H^0(X,\F\otimes L^{\otimes m})
$$
is surjective.
\end{enumerate}
\end{theorem}

We also recall the notion of normal generation.  A globally generated line bundle $\Mcal$ is normally
generated if the multiplication maps
$$
\operatorname{Sym}^n H^0(X,\Mcal)
\longrightarrow
H^0(X,\Mcal^{\otimes n})
$$
are surjective for all $n\ge 1$.  If $\Mcal$ is very ample, normal generation is the
degree-generation part of projective normality for the embedding defined by the complete linear
series $|\Mcal|$.  In the proof below, only the displayed surjectivity is used.

For later use, we record the asymptotic source of these properties.  If $A$ is ample on a normal
projective variety $X$, then $A^{\otimes m}$ is very ample and normally generated for all
$m\gg0$; moreover, for $m\gg0$ the embedding defined by $A^{\otimes m}$ is projectively
normal.  This follows from Serre vanishing together with the very-ampleness criterion and
Mumford's regularity theorem; see Lazarsfeld
\cite[Theorems~1.2.6 and~1.8.5]{Lazarsfeld}.

\subsection{The degree-\texorpdfstring{$c$}{c} ideal associated to divisorial data}

The next proposition is the point where the preceding conventions become essential.  It explains
why the ideal generated by $H^0(X,cL-D)$ behaves like the divisor $D$ after saturation, even
though it is generated in degree $c$.

\begin{proposition}\label{prop:generation-equality}
Let $X$ be a normal projective variety, let $L$ be an ample line bundle with standard graded
section ring $S=R(X,L)$, let $D$ be an effective Cartier divisor, and fix an integer $c\ge 1$.
Set
$$
\Mcal:=L^{\otimes c}\otimes \OO_X(-D).
$$
Equivalently, in the additive notation fixed above, $\Mcal=cL-D$.  Since $D$ is effective,
multiplication by the canonical section of $\OO_X(D)$ gives an inclusion
$$
H^0(X,\Mcal)=H^0(X,cL-D)
\hookrightarrow
H^0(X,cL)=S_c.
$$
Throughout this proposition, we regard $H^0(X,\Mcal)$ as a subspace of the degree-$c$ piece of
$S$ through this inclusion, and we define
$$
J:=S\cdot H^0(X,\Mcal)\subseteq S.
$$
Thus $J$ is the homogeneous ideal generated by this degree-$c$ vector space.  In particular, even
if $\Mcal\cong \OO_X$, the sections of $H^0(X,\Mcal)=H^0(X,cL-D)$ are placed in degree $c$, not in degree $0$.

Assume that the line bundle $\Mcal$ is globally generated.
\begin{enumerate}
\item The associated sheaf of $J$ is
$$
\widetilde J\cong \OO_X(-D).
$$
Consequently, for every $n\ge 1$ one has
$$
\widetilde{J^n}\cong \OO_X(-nD)
\qquad\text{and}\qquad
(J^n)^{\rm sat}=I_{nD}.
$$
\item The ideal $J$ is generated in degree $c$.  In particular,
$$
(J^n)_m=0
\qquad\text{for all }n\ge 1\text{ and }m<cn.
$$
\item If, in addition, $\Mcal$ is normally generated and $\Mcal^{\otimes n}$ is $0$-regular with respect to
$L$ for all $n\gg 0$, then for all $n\gg 0$ and all $m\ge cn$ one has
$$
(J^n)_m=H^0(X,mL-nD).
$$
\end{enumerate}
\end{proposition}

\begin{proof}
Because $\Mcal$ is globally generated, the evaluation map gives a surjection
$$
H^0(X,\Mcal)\otimes \OO_X
\twoheadrightarrow
\Mcal.
$$
Twisting by $L^{-c}$ gives
$$
H^0(X,\Mcal)\otimes L^{-c}
\twoheadrightarrow
\Mcal\otimes L^{-c}=\OO_X(-D).
$$
Under the inclusion
$$
H^0(X,\Mcal)\subseteq S_c=H^0(X,cL),
$$
these sections generate $J$ in degree $c$.  After passing to $\Proj S$, the degree shift by $c$
contributes the factor $L^{-c}$, and hence $\widetilde J$ is the image of the last displayed
map inside $\OO_X$.  Since the target is the invertible ideal sheaf $\OO_X(-D)$, we obtain
$$
\widetilde J\cong \OO_X(-D).
$$
Because $\OO_X(-D)$ is invertible,
$$
\widetilde{J^n}\cong (\widetilde J)^n\cong \OO_X(-nD).
$$
The graded ideal transform then gives
$$
(J^n)^{\rm sat}
=
\Gamma_*(\widetilde{J^n})
=
\Gamma_*(\OO_X(-nD))
=
I_{nD}.
$$
This proves part $(1)$.

Part $(2)$ follows directly from the definition: $J$ is generated by a subspace of $S_c$, so
every element of $J^n$ has degree at least $cn$.

Now assume the additional hypotheses.  Since $\Mcal$ is normally generated, the map
$$
\operatorname{Sym}^n H^0(X,\Mcal)
\twoheadrightarrow
H^0(X,\Mcal^{\otimes n})
$$
is surjective.  Therefore the degree-$cn$ piece of $J^n$ is
$$
(J^n)_{cn}
=
H^0(X,\Mcal^{\otimes n})
=
H^0(X,cnL-nD).
$$
Let $m\ge cn$.  For $n\gg 0$, the sheaf $\Mcal^{\otimes n}$ is $0$-regular with respect to $L$.
By Theorem~\ref{thm:regularity-package}, the multiplication map
$$
H^0(X,\Mcal^{\otimes n})\otimes H^0(X,(m-cn)L)
\twoheadrightarrow
H^0(X,\Mcal^{\otimes n}\otimes L^{\otimes(m-cn)})
=
H^0(X,mL-nD)
$$
is surjective.  The source maps into $(J^n)_m$, so
$$
H^0(X,mL-nD)\subseteq (J^n)_m.
$$
The reverse inclusion follows from $J\subseteq I_D$, hence
$$
J^n\subseteq I_D^n\subseteq I_{nD},
$$
and the degree-$m$ piece of $I_{nD}$ is exactly $H^0(X,mL-nD)$.  Thus
$$
(J^n)_m=H^0(X,mL-nD)
$$
for all $m\ge cn$ and all $n\gg 0$.
\end{proof}

\begin{remark}
The hypothesis that $\Mcal$ is globally generated is used only to identify the sheaf generated by the
degree-$c$ sections with $\OO_X(-D)$.  The normal-generation hypothesis controls the degree
$cn$ piece of $J^n$, and the eventual regularity hypothesis for $\Mcal^{\otimes n}$ then propagates this equality to all higher
degrees.  This separation of roles is important: global generation alone does not imply normal
generation.
\end{remark}

\subsection{Volumes, thresholds, and uniform asymptotics}

Let $X$ be a projective variety of dimension $r$.  For a Cartier divisor $A$ on $X$, its volume is
$$
\vol_X(A)
:=
\limsup_{m\to\infty}
\frac{h^0(X,mA)}{m^r/r!}.
$$
The definition is first made for Cartier divisors because Cartier divisors determine line bundles,
and the volume measures the asymptotic growth of spaces of global sections of these line
bundles.  The function $\vol_X$ depends only on the numerical class of $A$ and extends
continuously to $N^1(X)_\RR$; see Lazarsfeld \cite[Theorem~2.2.44]{Lazarsfeld}.  It is
homogeneous of degree $r$:
$$
\vol_X(aA)=a^r\vol_X(A)
\qquad
(a\in\RR_{\ge 0}),
$$
and if $A$ is nef, then
$$
\vol_X(A)=A^r.
$$

Now let $L$ be ample and let $D$ be an effective Cartier divisor on $X$.  We define the
pseudo-effective threshold of $D$ with respect to $L$ by
$$
\tau=\tau_L(D)
:=
\inf\{\,t\ge 0\mid tL-D\text{ is pseudo-effective}\,\}.
$$
This number is the first point at which the ray $tL-D$ meets the pseudo-effective cone.  In
the one-ideal setting below, $cL-D$ is globally generated and hence
pseudo-effective, so $\tau\leq c$.  In particular, if $m/n<\tau$, then
$$
\frac{m}{n}L-D
$$
is not pseudo-effective.  Hence $mL-nD$ is not effective, and therefore
$$
H^0(X,mL-nD)=0.
$$

For the one-ideal formula, the relevant range is the strip
$$
\tau n\le m<cn.
$$
In this strip one may write
$$
mL-nD
=
n\left(\frac{m}{n}L-D\right).
$$
Thus the leading asymptotic behavior of $h^0(X,mL-nD)$ is governed by the volume of the
class
$$
\frac{m}{n}L-D.
$$
More precisely, on the compact segment
$$
K:=\{\,tL-D\mid \tau\le t\le c\,\}\subset N^1(X)_\RR,
$$
we use the standard uniform asymptotic estimate for sections.  There is one
minor endpoint issue when $\tau=0$: the strip then contains the single term
$m=0$, whereas the positive-coefficient form of
\cite[Proposition~3.5.1]{BGGJKM} is applied only for $m\geq1$.  We separate
this term.  Since $D$ is effective, $\OO_X(-nD)\subseteq\OO_X$, and hence
$$
h^0(X,-nD)\leq h^0(X,\OO_X)=O(1).
$$
For the remaining integers $m$ with $\tau n\leq m<cn$ and $m\geq1$, apply
\cite[Proposition~3.5.1]{BGGJKM} to the two fixed divisors $L$ and $-D$.
Writing $\rho_0(N)=o(N^r)$ for the error function supplied there, one obtains
$$
\left|
h^0(X,mL-nD)
-
\frac{n^r}{r!}\vol_X\left(\frac{m}{n}L-D\right)
\right|
\leq \rho_0(m+n).
$$
Set
$$
e(N):=\frac{\rho_0(N)}{N^r}
\qquad\text{and}\qquad
\widehat\rho(n):=(c+1)^r n^r\sup_{N\geq n}e(N).
$$
Since $e(N)\to0$, one has $\widehat\rho(n)=o(n^r)$.  Moreover,
$n\leq m+n\leq(c+1)n$ throughout the strip, and hence
$$
\rho_0(m+n)\leq\widehat\rho(n)
$$
uniformly in $m$.  The separated endpoint contributes only $O(1)$, and consequently
$$
\sum_{\tau n\le m<cn}h^0(X,mL-nD)
=
\frac{n^r}{r!}
\sum_{\substack{\tau n\le m<cn\\ m\geq1}}
\vol_X\left(\frac{m}{n}L-D\right)
+
o(n^{r+1}).
$$
After dividing by $n^{r+1}$, the remaining expression is the Riemann sum
$$
\frac1n
\sum_{\tau n\le m<cn}
\vol_X\left(\frac{m}{n}L-D\right),
$$
which converges, by continuity of the volume function, to
$$
\int_{\tau}^{c}\vol_X(tL-D)\,dt.
$$
The possible rounding at the endpoints contributes only $O(1)$ summands and does not affect
the limit.

\subsection{A logarithmic transcendence input}

We shall use the following form of Baker's theorem on linear forms in logarithms; see
Waldschmidt \cite[Theorem~11.1]{Waldschmidt}.  Throughout, a logarithm of an algebraic
number means a chosen complex logarithm.

\begin{theorem}\label{thm:Baker}
Let $\ell_1,\dots,\ell_m\in\CC$ be $\QQ$-linearly independent logarithms of algebraic numbers.
Then
$$
1,\ell_1,\dots,\ell_m
$$
are linearly independent over $\overline{\QQ}$.
\end{theorem}

We shall use the following immediate consequence.

\begin{corollary}\label{cor:Baker-linear-form}
Let $\ell_1,\dots,\ell_N$ be logarithms of algebraic numbers, and let
$A_0,A_1,\dots,A_N\in\overline{\QQ}$.  Put
$$
\Lambda
=
A_0+\sum_{j=1}^N A_j\ell_j.
$$
If the logarithmic part
$$
\sum_{j=1}^N A_j\ell_j
$$
is nonzero, then $\Lambda$ is transcendental.
\end{corollary}

\begin{proof}
Choose a maximal $\QQ$-linearly independent subfamily $u_1,\dots,u_s$ among the logarithms
that occur.  Every $\ell_j$ is a $\QQ$-linear combination of the $u_i$, so the logarithmic part can
be rewritten as
$$
\sum_{i=1}^s B_i u_i
$$
with $B_i\in\overline{\QQ}$.  By assumption, not all $B_i$ are zero.  If $\Lambda$ were algebraic,
then
$$
(A_0-\Lambda)+\sum_{i=1}^s B_i u_i=0
$$
would be a nontrivial $\overline{\QQ}$-linear relation among
$$
1,u_1,\dots,u_s,
$$
contradicting Theorem~\ref{thm:Baker}.  Hence $\Lambda$ is transcendental.
\end{proof}

In Section~\ref{sec:arithmetic-input}, this corollary is applied to a finite
$\overline{\QQ}$-linear combination of real logarithms of positive algebraic
numbers.  The point is that one does not need the displayed logarithms to be
$\QQ$-linearly independent.  It is enough to pass to a maximal independent
subfamily and verify directly that the logarithmic part of the final expression
is nonzero.

\section{From filtrations to a one-ideal volume formula}\label{sec:from-filtrations-to-one-ideal}

The purpose of this section is to explain the passage from a flexible asymptotic construction to the
single-ideal formula that drives the rest of the paper.  The first subsection gives a filtration whose
epsilon multiplicity is the transcendental number $\pi$.  This example is deliberately elementary:
it shows that transcendence is not foreign to the asymptotic length expressions defining epsilon
multiplicity.  Its limitation is equally important, however, because it is only a filtration-level
construction.  The second subsection supplies the mechanism that replaces such auxiliary
filtrations by powers of one homogeneous ideal.  In a divisorial section-ring setting, the saturation
filtration of a degree-$c$ generated ideal agrees with a divisorial filtration, and the resulting
epsilon multiplicity becomes a divisor-volume integral.

\subsection{A filtration-level model}\label{sec:filtration-example}\label{subsec:filtration-model}

We begin with a filtration whose epsilon multiplicity is $\pi$.  This is not yet the main theorem,
but it clarifies why transcendence appears naturally before one passes to the epsilon multiplicity
of a single ideal; compare \cite{CutSarkar}.

Let
$$
R=\CC[x,y]_{(x,y)},\qquad \m=(x,y),\qquad \alpha:=\pi.
$$
Set $J_0:=R$, and for $n\ge 1$ define the monomial ideal
$$
J_n:=\left(x^a y^b \mid a,b\in \NN,\ a+\frac{b}{\alpha}\ge n\right)\subset R.
$$

\begin{proposition}\label{prop:transcendental-filtration}
The family $\mathcal J=\{J_n\}_{n\ge 0}$ is a filtration of $\m$-primary ideals, satisfies the
$A(c)$ condition of Cutkosky--Sarkar, and
$$
2!\lim_{n\to\infty}\frac{\lambda_R(R/J_n)}{n^2}=\pi.
$$
In particular, the multiplicity of the filtration $\mathcal J$ is $\pi$.
\end{proposition}

\begin{proof}
The defining inequalities immediately give $J_{n+1}\subseteq J_n$ for every $n\geq0$.  If
$x^a y^b\in J_m$ and $x^c y^d\in J_n$, then
$$
(a+c)+\frac{b+d}{\alpha}
=
\left(a+\frac{b}{\alpha}\right)+\left(c+\frac{d}{\alpha}\right)\ge m+n,
$$
so $J_mJ_n\subseteq J_{m+n}$.

Each $J_n$ is $\m$-primary because $x^n\in J_n$ and $y^{\lceil \alpha n\rceil}\in J_n$.
Choose $c=\lceil \alpha\rceil=4$.  If a monomial $x^a y^b\in \m^{cn}$, then $a+b\ge cn$, hence
$$
a+\frac{b}{\alpha}\ge \frac{a+b}{\alpha}\ge \frac{cn}{\alpha}\ge n.
$$
Thus $\m^{cn}\subseteq J_n$, and therefore
$$
(J_n:\m^\infty)\cap \m^{cn}=J_n\cap \m^{cn},
$$
which is exactly property $A(c)$ in the sense of \cite{CutSarkar}.

Finally, $\lambda_R(R/J_n)$ is the number of lattice points $(a,b)\in \NN^2$ satisfying
$$
a+\frac{b}{\alpha}<n.
$$
Thus
$$
\lambda_R(R/J_n)
=
\sum_{a=0}^{n-1}\#\{\,b\in \NN \mid b<\alpha(n-a)\,\}
=
\sum_{j=1}^{n}\lceil \alpha j\rceil
=
\frac{\alpha}{2}n^2+O(n).
$$
Hence
$$
\lim_{n\to\infty}\frac{\lambda_R(R/J_n)}{n^2}
=
\frac{\alpha}{2},
$$
and multiplying by $2!$ gives $\pi$.
\end{proof}

\begin{remark}
Proposition~\ref{prop:transcendental-filtration} shows that transcendental behavior is already
visible at the level of filtrations.  The main difficulty is to realize such arithmetic behavior from
the ordinary powers of a single ideal.  The next subsection gives the structural mechanism that makes
this possible.
\end{remark}

\subsection{A one-ideal volume formula}\label{subsec:one-ideal-volume-formula}

We now prove the integral formula that drives the paper. Our formula states that, under suitable positivity and regularity
hypotheses, the epsilon multiplicity of the localized ideal generated by $H^0(X,cL-D)$ in degree
$c$ becomes an integral of a divisor-volume function.  In other words, the local invariant, epsilon multiplicity,
is converted into an explicitly geometric quantity.

\begin{theorem}\label{thm:one-ideal}
Let $X$ be a normal projective variety of dimension $r$ over $\mathbb{C}$, let $L$ be ample with standard graded
section ring $R(X,L)$, and let $D$ be an effective Cartier divisor.  Fix an integer $c\ge 1$, and set
$$
\Mcal:=L^{\otimes c}\otimes \OO_X(-D).
$$
Equivalently, in the additive notation fixed above, $\Mcal=cL-D$.  As in Proposition~\ref{prop:generation-equality}, we view $H^0(X,\Mcal)=H^0(X,cL-D)$ inside the degree-$c$ piece
$$
R(X,L)_c=H^0(X,cL)
$$
via multiplication by the canonical section of $\OO_X(D)$, and set
$$
J:=R(X,L)\cdot H^0(X,\Mcal)\subseteq R(X,L).
$$
Thus $J$ is generated in degree $c$.  Assume:
\begin{enumerate}
\item $\Mcal$ is globally generated;
\item $\Mcal$ is normally generated;
\item for all $n\gg 0$, $\Mcal^{\otimes n}$ is $0$-regular with respect to $L$.
\end{enumerate}
Let
$$
\tau=\tau_L(D):=\inf\{\,t\ge 0\mid tL-D\ \text{is pseudo-effective}\,\}.
$$
Since $\Mcal=cL-D$ is globally generated, it is pseudo-effective, and therefore
$\tau\leq c$.  Then
$$
\eps(J R_{R_+})
=
(r+1)\int_{\tau}^{c}\vol_X(tL-D)\,dt.
$$
\end{theorem}

\begin{proof}
By Proposition~\ref{prop:generation-equality},
$$
(J^n)^{\mathrm{sat}}=I_{nD}
$$
for every $n\ge 1$.  Since $I_{nD}/J^n$ is a finite-length module supported at $R_+$, its length
is unchanged after localizing at $R_+$.  Hence
$$
\eps(J R_{R_+})
=
(r+1)!\limsup_{n\to\infty}\frac{\lambda(I_{nD}/J^n)}{n^{r+1}}.
$$

Again by Proposition~\ref{prop:generation-equality}, the ideal $J$ is generated in degree $c$, so
$$
(J^n)_m=0\qquad \text{for }m<cn.
$$
Also, if $m<\tau n$, then $mL-nD$ is not effective by definition of $\tau$, and therefore
$$
(I_{nD})_m=H^0(X,mL-nD)=0.
$$

On the other hand, Proposition~\ref{prop:generation-equality} shows that for all $n\gg 0$, and all
$m\ge cn$ one has
$$
(J^n)_m=(I_{nD})_m.
$$

Thus
$$
(I_{nD}/J^n)_m=0
\qquad \text{unless}\qquad
\tau n\le m<cn \text{ and }n\gg 0.
$$

Thus the entire defect is concentrated in the strip $\tau n\le m<cn$.

Therefore, for $n\gg 0$,
$$
\lambda(I_{nD}/J^n)
=
\sum_{\tau n\le m<cn}\dim_\CC (I_{nD}/J^n)_m
=
\sum_{\tau n\le m<cn} h^0(X,mL-nD),
$$
because $(J^n)_m=0$ throughout this strip.

If $\tau=0$, set
$$
\delta_n:=h^0(X,-nD),
$$
and otherwise set $\delta_n:=0$.  Since $D$ is effective,
$\OO_X(-nD)\subseteq\OO_X$; as $X$ is integral and projective over
$\CC$, this gives $0\leq\delta_n\leq h^0(X,\OO_X)=1$.  Hence
$$
\lambda(I_{nD}/J^n)
=
\delta_n+
\sum_{\substack{\tau n\le m<cn\\m\geq1}}
h^0(X,mL-nD).
$$

Now set
$$
K:=\{\,tL-D\mid \tau\le t\le c\,\}\subset N^1(X)_\RR.
$$
Apply \cite[Proposition~3.5.1]{BGGJKM} to the two fixed divisors $L$
and $-D$, and let $\rho_0(N)=o(N^r)$ be the resulting error function in
the total coefficient.  For every $n\geq1$ and every $m$ with
$\tau n\leq m<cn$ and $m\geq1$, one then has
$$
\left|
h^0(X,mL-nD)-\frac{n^r}{r!}\vol_X\!\left(\frac{m}{n}L-D\right)
\right|
\leq \rho_0(m+n).
$$
Put
$$
e(N):=\frac{\rho_0(N)}{N^r}
\qquad\text{and}\qquad
\widehat\rho(n):=(c+1)^r n^r\sup_{N\geq n}e(N).
$$
Then $\widehat\rho(n)=o(n^r)$.  Since
$n\leq m+n\leq(c+1)n$ throughout the range, we also have
$\rho_0(m+n)\leq\widehat\rho(n)$ uniformly in $m$.  Hence
$$
\left|
\lambda(I_{nD}/J^n)
-
\frac{n^r}{r!}\sum_{\substack{\tau n\le m<cn\\m\geq1}}
\vol_X\!\left(\frac{m}{n}L-D\right)
\right|
\le
\delta_n+(cn+1)\widehat\rho(n)
=
o(n^{r+1}).
$$
Dividing by $n^{r+1}$, we obtain
$$
    \frac{\lambda(I_{nD}/J^n)}{n^{r+1}}
=
\frac1{r!}\cdot \frac1n
\sum_{\substack{\tau n\le m<cn\\m\geq1}}
\vol_X\!\left(\frac{m}{n}L-D\right)+o(1).
$$
Omitting the single index $m=0$ when $\tau=0$ does not change the limit.
Since the volume function is continuous on $N^1(X)_\RR$, the right-hand side
is a Riemann sum for
$$
\frac1{r!}\int_{\tau}^{c}\vol_X(tL-D)\,dt.
$$
Thus, taking the limit superior in the displayed equality above and multiplying by $(r+1)!$ gives the asserted formula.
\end{proof}

The hypotheses of Theorem~\ref{thm:one-ideal} are automatic after a
sufficiently positive common rescaling whenever $cL-D$ is ample.

\begin{theorem}\label{thm:realization}
Let $X$ be a normal projective variety of dimension $r$ over $\mathbb C$, let
$L$ be an ample line bundle, let $D$ be an effective Cartier divisor, and let
$c\geq1$ be an integer such that $cL-D$ is ample.  For every sufficiently
large integer $k$, set
$$
S_k=R(X,kL)
$$
and place the vector space $H^0(X,k(cL-D))$ in degree $c$ of $S_k$.  Define
$$
J_k=S_k\cdot H^0(X,k(cL-D)).
$$
Then $S_k$ is a normal standard graded domain and
$$
\eps\bigl(J_k(S_k)_{(S_k)_+}\bigr)
=
(r+1)k^r
\int_{\tau_L(D)}^c\vol_X(tL-D)\,dt.
$$
\end{theorem}

\begin{proof}
Put $\Mcal=cL-D$.  For $k\gg0$, the line bundle $kL$ is very ample and gives
a projectively normal embedding, while $k\Mcal$ is globally generated and
normally generated.  These follow from Serre vanishing and Mumford's
regularity theorem; see
\cite[Theorems~1.2.6 and~1.8.5]{Lazarsfeld}.  Since $X$ is
normal, projective normality implies that $S_k=R(X,kL)$ is a normal standard
graded domain.

For each $i=1,\ldots,r$, Serre vanishing applied to the ample line bundle
$k\Mcal$ and the fixed coherent sheaf $(kL)^{-i}$ gives
$$
H^i\bigl(X,(k\Mcal)^{\otimes n}\otimes(kL)^{-i}\bigr)=0
\qquad(n\gg0).
$$
Thus $(k\Mcal)^{\otimes n}$ is $0$-regular with respect to $kL$ for all
$n\gg0$.  Theorem~\ref{thm:one-ideal} applies to the triple
$(X,kL,kD)$ with the same integer $c$.

The pseudo-effective threshold is unchanged by the common scaling:
$$
\tau_{kL}(kD)=\tau_L(D).
$$
Since the volume on an $r$-dimensional variety is homogeneous of degree $r$,
$$
\vol_X\bigl(t(kL)-kD\bigr)
=
k^r\vol_X(tL-D).
$$
Substituting these identities into Theorem~\ref{thm:one-ideal} proves the
formula.
\end{proof}

\begin{corollary}\label{cor:transcendence-transport}
In the setting of Theorem~\ref{thm:realization}, if
$$
\int_{\tau_L(D)}^c\vol_X(tL-D)\,dt
$$
is transcendental, then, for every sufficiently large $k$, the normal standard
graded domain $S_k$ contains a homogeneous ideal $J_k$ whose localization at
$(S_k)_+$ has transcendental epsilon multiplicity.
\end{corollary}

\begin{remark}\label{rem:degree-convention}
The theorem reduces the epsilon multiplicity of the single ideal
$$
J=R(X,L)\cdot H^0(X,\Mcal)=R(X,L)\cdot H^0(X,cL-D)
$$
to a divisor-volume integral.  Here the generating space is always interpreted as lying in degree
$c$ of $R(X,L)$; this convention is essential, for example, when $cL-D\sim 0$.  From this point
onward, the problem is geometric and arithmetic.
\end{remark}

\section{A shifted projective-bundle construction}\label{sec:BN}

The purpose of this section is to construct geometric data to which Theorem~\ref{thm:one-ideal} applies while retaining the disk geometry suggested by the numerical framework of Borntr\"ager and Nickel.  We first recall the circular nef cone of a suitable abelian surface and then fix explicitly the triangle determined by three divisor classes.  We then shift these classes only in the ample direction.  This improves positivity without changing their transverse coordinates, and hence leaves the triangle unchanged.  The resulting projective bundle produces a moving disk whose intersection with this fixed triangle controls the divisor volume.  Finally, after passing to a sufficiently positive multiple, we verify the global generation, normal generation, and regularity hypotheses required by Theorem~\ref{thm:one-ideal}.

The numerical mechanism is concrete.  On the abelian surface, the nef condition is a disk inequality in the two transverse coordinates.  On the associated split projective bundle, the volume formula turns the simplex of symmetric powers into the fixed triangle
$$
T=\operatorname{conv}\{(-2,-2),(2,-2),(4,0)\}\subset \RR^2.
$$
The construction therefore separates into two tasks: preserve this triangle, and move the line bundles far enough into the ample cone for the one-ideal formula to apply.  The shift below accomplishes both.

\subsection{The circular nef cone on the abelian surface}\label{subsec:BN-surface}

We begin with the surface geometry that produces the moving disk.  The required input from \cite[Proposition~14]{BN} is a numerical basis in which the nef cone becomes circular and the volume function is given by a quadratic form.

Let $E$ be an elliptic curve over $\mathbb{C}$ without complex multiplication, and let $Y=E\times E$ be the abelian surface used in \cite[Proposition~14]{BN}.  Following the numerical normalization in the same proposition, we choose divisor classes
$A,D_1,D_2\in N^1(Y)$ whose images form an $\RR$-basis of $N^1(Y)_\RR$ and for which the intersection form is diagonal:
$$
A^2=N,
\qquad
D_1^2=D_2^2=-N,
\qquad
A\cdot D_i=D_1\cdot D_2=0,
$$
where
$$
N:=A^2>0.
$$
Thus $N$ is a fixed positive integer determined by the chosen polarization class $A$ on $Y$.  In these coordinates the nef cone is circular:
$$
uA+xD_1+yD_2\text{ is nef}
\quad\Longleftrightarrow\quad
u\ge \sqrt{x^2+y^2}.
$$
The strict inequality gives ampleness.  Thus
$$
uA+xD_1+yD_2\text{ is ample}
\quad\Longleftrightarrow\quad
u>\sqrt{x^2+y^2}.
$$
For a class $B=uA+xD_1+yD_2$, the above intersection form gives
$$
B^2=N(u^2-x^2-y^2).
$$
On an abelian variety, the pseudo-effective and nef cones coincide.  Indeed, every
effective divisor is nef: for a given irreducible curve, translate the divisor so that it
does not contain the curve and use translation invariance of numerical equivalence.
Conversely, ample classes are big and hence pseudo-effective; taking closures gives the
reverse inclusion.  See also \cite[Proposition~14]{BN} in the present numerical setting.
Consequently, a class outside the displayed circular cone is not pseudo-effective and has
volume zero.  Since nef
divisors have volume equal to their top self-intersection, the volume on $Y$ is therefore
$$
\vol_Y(uA+xD_1+yD_2)
=
\begin{cases}
N(u^2-x^2-y^2),& u\ge 0\text{ and }x^2+y^2\le u^2,\\[4pt]
0,& \text{otherwise.}
\end{cases}
$$
This formula is the source of the disks that appear later: the condition $x^2+y^2\le u^2$ cuts out a disk in the $(D_1,D_2)$-plane.

\subsection{A triangle inspired by Borntr\"ager--Nickel}\label{subsec:BN-triangle}

We next pass from the circular cone on $Y$ to the fixed planar region that controls the projective-bundle volume.  Motivated by the numerical projective-bundle setup of \cite[Proposition~14]{BN}, we fix the following three divisor classes:
$$
M_0=A-2D_1-2D_2,
\qquad
M_1=A+2D_1-2D_2,
\qquad
M_2=A+4D_1.
$$
The $A$-coefficients of the three divisors are all equal to $1$.  Their projections to the subspace
$$
\operatorname{span}_{\mathbb R}\{D_1,D_2\}
$$
are
$$
-2D_1-2D_2,
\qquad
2D_1-2D_2,
\qquad
4D_1.
$$
Equivalently, writing
$$
M_i=A+v_i
$$
with
$$
v_i\in \operatorname{span}_{\mathbb R}\{D_1,D_2\},
$$
the corresponding coefficient vectors are
$$
(-2,-2),
\qquad
(2,-2),
\qquad
(4,0).
$$
We denote their convex hull by
$$
T:=\operatorname{conv}\{(-2,-2),(2,-2),(4,0)\}\subset \RR^2.
$$
This triangle is the fixed transverse region in our construction: it is the domain over which the disk-controlled volume calculation will be carried out below.

We use the quotient convention for projective bundles:
$$
\mathbf P_Y(E):=\Proj_Y \Sym^\bullet E.
$$
Thus $\mathbf P_Y(E)$ parametrizes one-dimensional quotients of the fibers of $E$, the tautological line bundle satisfies
$$
\pi^*E\twoheadrightarrow \OO_{\mathbf P_Y(E)}(1),
$$
and
$$
\pi_*\OO_{\mathbf P_Y(E)}(m)=\Sym^m E
\qquad (m\geq 0).
$$

See \cite[Tags~01OB and 01OC]{Stacks} for more details.

Let
$$
X=\mathbf P\bigl(\OO_Y(M_0)\oplus \OO_Y(M_1)\oplus \OO_Y(M_2)\bigr),
\qquad
\xi=\OO_X(1).
$$
By this convention, if $\pi:X\to Y$ denotes the projection, then
$$
\pi_*\OO_X(m\xi)
=
\Sym^m\bigl(\OO_Y(M_0)\oplus \OO_Y(M_1)\oplus \OO_Y(M_2)\bigr)
=
\bigoplus_{n_0+n_1+n_2=m}\OO_Y(n_0M_0+n_1M_1+n_2M_2).
$$
Thus sections of multiples of $\xi$ decompose into summands indexed by triples
$$
(n_0,n_1,n_2)\in\ZZ_{\geq 0}^3,
\qquad
n_0+n_1+n_2=m.
$$
After rescaling by $m$, the normalized triples
$\lambda_i=n_i/m$ become points of the simplex
$$
\Delta=\{(\lambda_0,\lambda_1,\lambda_2)\in\RR_{\ge 0}^3\mid
\lambda_0+\lambda_1+\lambda_2=1\}.
$$
The affine map
$$
\Delta\longrightarrow \RR^2,
\qquad
\lambda\longmapsto
\lambda_0(-2,-2)+\lambda_1(2,-2)+\lambda_2(4,0)
$$
has image exactly $T$.  Consequently, the volume of a tautological divisor
on a split projective bundle with these transverse classes is governed by an
integral over $T$ of the circular volume function from the preceding
subsection.

The circular coordinates and the split projective-bundle strategy are
inspired by \cite{BN}.  The three classes displayed above are, from this point
onward, part of the present construction.  We do not invoke any geometric
description of their convex hull from \cite{BN}: every assertion about $T$
and every volume formula needed in the proof is derived directly in
Section~\ref{sec:arithmetic-input}.  In particular, the distance from the
origin to $T$ is $4/\sqrt{10}>1$.  Thus a radial parameter equal to $1$ does
not meet the triangle; increasing the common $A$-coefficient is not merely a
positivity convenience but is also necessary for the disk to reach the
transverse region that contributes to our integral.

For $xD_1+yD_2\in\operatorname{span}_{\RR}\{D_1,D_2\}$, write
$$
\|xD_1+yD_2\|:=\sqrt{x^2+y^2}.
$$
The radii of the three vertices of $T$ are
$$
\|(-2,-2)\|=\|(2,-2)\|=2\sqrt2,
\qquad
\|(4,0)\|=4.
$$
Equivalently,
$$
\|M_0-A\|=\|M_1-A\|=2\sqrt2,
\qquad
\|M_2-A\|=4.
$$
These numbers later reappear as critical radii for the moving disk.

\subsection{A positivity shift preserving the triangle}\label{subsec:why-shift}

The unshifted classes above produce the required disk geometry, but they do not by themselves provide the positivity needed for the one-ideal formula.  We therefore move all three classes in the common ample direction $A$.  This places the corresponding line bundles deeper in the ample cone while leaving their $D_1$ and $D_2$ coordinates unchanged.  The shift improves the algebraic positivity without altering the triangle that controls the volume computation.

Indeed, if we write $M_i=A+v_i$ with
$$
v_i\in \operatorname{span}_{\mathbb R}\{D_1,D_2\},
$$
then
$$
M_i+sA=(s+1)A+v_i.
$$
Thus the shift changes only the $A$-coordinate and leaves the projection to
$$
\operatorname{span}_{\mathbb R}\{D_1,D_2\}
$$
unchanged.  Hence the three coefficient vectors, and therefore the triangle $T$, remain the same.  The positivity improves, while the disk--triangle geometry remains unchanged.

\subsection{The shifted projective bundle and the moving disk}

We now carry out the shift and introduce the projective bundle on which the one-ideal construction will live.  The first point is that, once the common $A$-component is sufficiently large, the tautological line bundle becomes both ample and globally generated.

Fix an integer $s\geq 5$, and define
$$
\mathcal E_s
=
\OO_Y(M_0+sA)\oplus \OO_Y(M_1+sA)\oplus \OO_Y(M_2+sA),
\qquad
X_s:=\mathbf P(\mathcal E_s).
$$
Let
$$
M_s:=\OO_{X_s}(1).
$$

\begin{proposition}\label{prop:Ms-positive}
For $s\geq 5$, the line bundle $M_s$ is ample and globally generated.
\end{proposition}

\begin{proof}
For each $i$, write
$$
M_i=A+v_i,
\qquad
v_i\in\{-2D_1-2D_2,\ 2D_1-2D_2,\ 4D_1\}.
$$
The norm of each $v_i$ with respect to the circular nef cone is at most $4$.
Thus
$$
M_i+sA=(s+1)A+v_i
$$
is ample for $s\geq 5$, because $s+1>4$.

Moreover,
$$
M_i+sA=A+\bigl(M_i+(s-1)A\bigr),
$$
and
$$
M_i+(s-1)A=sA+v_i
$$
is ample for $s\geq 5$, since $s>4\geq \|v_i\|$. Hence
$$
\OO_Y(M_i+sA)
\cong
\OO_Y(A)\otimes \OO_Y(M_i+(s-1)A)
$$
is the tensor product of two ample line bundles on the abelian variety $Y$.
By Bauer--Szemberg \cite[Theorem~1.1(a)]{BS}, this tensor product is globally
generated. Therefore each $\OO_Y(M_i+sA)$ is globally generated.

It follows that
$$
\mathcal E_s=\OO_Y(M_0+sA)\oplus \OO_Y(M_1+sA)\oplus \OO_Y(M_2+sA)
$$
is globally generated. Since each summand is ample, $\mathcal E_s$ is an ample
vector bundle by Hartshorne \cite[Proposition~2.2]{Hartshorne}.

Let
$$
\pi:X_s=\mathbf P(\mathcal E_s)\to Y.
$$
By the definition of global generation, there is a surjection
$$
H^0(Y,\mathcal E_s)\otimes \OO_Y \twoheadrightarrow \mathcal E_s;
$$
see \cite[Tag~01AL]{Stacks}. Pulling back along $\pi$, we obtain
a surjection
$$
H^0(Y,\mathcal E_s)\otimes \OO_{X_s}\twoheadrightarrow \pi^*\mathcal E_s.
$$
On the other hand, for a projective bundle the tautological invertible sheaf
comes with a canonical surjection
$$
\pi^*\mathcal E_s\twoheadrightarrow \OO_{X_s}(1),
$$
and in the Stacks Project this surjectivity is exactly
\cite[Tag~01O4]{Stacks}; see also the discussion after
\cite[Tag~01OB]{Stacks}. Composing the two surjections gives
$$
H^0(Y,\mathcal E_s)\otimes \OO_{X_s}\twoheadrightarrow \OO_{X_s}(1).
$$
Hence $\OO_{X_s}(1)=M_s$ is globally generated by
\cite[Tag~01AL]{Stacks}.

Finally, since $\mathcal E_s$ is ample, the tautological line bundle
$\OO_{X_s}(1)$ is ample on $\mathbf P(\mathcal E_s)$; see
\cite[Proposition~3.2]{Hartshorne}. Therefore $M_s$ is ample and globally
generated.
\end{proof}

We next choose an ample class $L_0$ and an effective divisor $D_0$ so that the endpoint of the ray $tL_0-D_0$ is tautological, while the classes with $t<q$ retain the same transverse coordinates.  This is the point at which the fixed triangle becomes a moving disk problem.

Fix an integer $\beta\geq 2$.  The line bundles $A$ and $(\beta-1)A$ are
ample, so Bauer--Szemberg \cite[Theorem~1.1(a)]{BS} implies that their tensor
product $\beta A$ is globally generated.  Choose an effective divisor
$$
A_\beta\in |\beta A|
$$
and set
$$
B:=\pi^*A_\beta,\qquad L_0:=M_s+B.
$$
Thus $B$ is an effective Cartier divisor and
$$
B\equiv\pi^*(\beta A)
$$
numerically.  Since $B$ is nef, the divisor $L_0=M_s+B$ is ample.  Moreover,
$M_s$ and $B$ are globally generated, so $L_0$ is globally generated.

For an integer $q>4$, set
$$
E_i(q):=(q-i)M_s-iB
\qquad (i=1,2,3,4).
$$
We choose $q$ sufficiently large so that all four divisors $E_i(q)$ are ample.
To see this, observe that
$$
\frac{1}{q}E_i(q)
=
\left(1-\frac{i}{q}\right)M_s-\frac{i}{q}B.
$$
For each fixed $i$, this class converges to $M_s$ in $N^1(X_s)_{\mathbb R}$ as
$q\to\infty$. Since $M_s$ is ample and the ample cone is open, the class
$q^{-1}E_i(q)$ is ample for all sufficiently large $q$. Hence $E_i(q)$ is ample
for all sufficiently large $q$. Since there are only finitely many indices
$i=1,2,3,4$, one may choose a single $q$ for which all four $E_i(q)$ are ample.

Define
$$
D_0:=qB,
\qquad
\tau_0:=\tau_{L_0}(D_0).
$$
Then
$$
qL_0-D_0=qM_s.
$$

\begin{remark}
The identity $qL_0-D_0=qM_s$ specifies the chosen positive endpoint $t=q$ of the
ray segment used in the construction.  At this endpoint the divisor is a positive
tautological divisor, which is later used to verify the positivity and regularity
hypotheses of Theorem~\ref{thm:one-ideal}.  The actual pseudo-effective boundary
occurs at $t=\tau_0<q$.  For $t<q$, one has
$$
tL_0-D_0\equiv tM_s-\pi^*(\beta(q-t)A).
$$
Thus, moving left from the endpoint $t=q$ subtracts a positive multiple of
the pullback of $A$ from the base direction.  In the projective bundle volume
formula this lowers the $A$-coefficient while leaving the transverse
$(D_1,D_2)$ coordinates unchanged.  Consequently the circular nef condition
on $Y$ becomes the moving-disk condition
$$
x^2+y^2\leq R(t)^2,
\qquad
R(t)=s+1+\beta-\frac{\beta q}{t}.
$$
Thus this choice simultaneously provides the positivity needed for the
one-ideal formula and preserves the disk geometry used in the arithmetic
calculation.
\end{remark}

\subsection{Verification of the one-ideal hypotheses}

The preceding construction supplies the required divisor volume geometry.  It remains to place it in the exact algebraic setting of Theorem~\ref{thm:one-ideal}.  Passing to a sufficiently positive common multiple gives a standard graded section ring, normal generation, and the regularity needed to identify the saturation defect with the divisor volume integral.

Choose an integer $k\gg 0$, and set
$$
L':=kL_0,\qquad D':=kD_0,\qquad \Mcal':=\OO_{X_s}(qL'-D')\cong \OO_{X_s}(kqM_s).
$$

\begin{proposition}\label{prop:geometric-hypotheses}
Assume that $q$ has been chosen sufficiently large so that the divisors
$$
E_i(q):=(q-i)M_s-iB
\qquad (i=1,2,3,4)
$$
are ample.  Then, for $k$ sufficiently large, the triple
$(X_s,L',D')$ satisfies all hypotheses of Theorem~\ref{thm:one-ideal}, where
$$
L':=kL_0,
\qquad
D':=kD_0,
\qquad
\Mcal':=\OO_{X_s}(qL'-D')\cong \OO_{X_s}(kqM_s),
$$
and the integer $c$ in Theorem~\ref{thm:one-ideal} is $c=q$.  Moreover,
$R(X_s,L')$ is a normal standard graded domain.
\end{proposition}

\begin{proof}
First note that $D'=kD_0$ is an effective Cartier divisor, since $D_0=qB$ and
$B$ was chosen to be an effective Cartier divisor.  Also $L'$ is ample because
$L_0$ is ample.

By asymptotic normal generation and projective normality for powers of ample
line bundles, after replacing $k$ by a larger integer if necessary, we may assume
that $L'=kL_0$ is very ample and defines a projectively normal embedding, while
$\Mcal'\cong \OO_{X_s}(kqM_s)$ is globally generated and normally generated;
see Lazarsfeld \cite[Theorems~1.2.6 and~1.8.5]{Lazarsfeld}.  In particular,
$R(X_s,L')$ is generated in degree one.  Since $X_s$ is smooth, hence normal,
and the embedding defined by $L'$ is projectively normal, its homogeneous
coordinate ring $R(X_s,L')$ is integrally closed.  Thus $R(X_s,L')$ is a normal
standard graded domain.

It remains to verify the required eventual regularity hypothesis, namely that
$(\Mcal')^{\otimes n}$ is $0$-regular with respect to $L'$ for all $n\gg 0$.
Since $\dim X_s=4$, this amounts to proving the vanishings
$$
H^j\bigl(X_s,(\Mcal')^{\otimes n}\otimes (L')^{-j}\bigr)=0
\qquad
(j=1,2,3,4)
$$
for all $n\gg 0$.  In additive notation, for $j=1,2,3,4$ one has
$$
n(qL'-D')-jL'
=
k\big((nq-j)M_s-jB\big)
=
kE_j(q)+(n-1)kqM_s.
$$
Here $E_j(q)$ is ample by the choice of $q$, and $M_s$ is ample by
Proposition~\ref{prop:Ms-positive}.  Thus, for each fixed $j$, the line bundle
$$
\OO_{X_s}(kE_j(q))\otimes \OO_{X_s}((n-1)kqM_s)
$$
is a fixed coherent sheaf twisted by sufficiently large powers of the ample
line bundle $\OO_{X_s}(kqM_s)$ as $n\to\infty$.  By Serre vanishing,
$$
H^j\bigl(X_s,(\Mcal')^{\otimes n}\otimes (L')^{-j}\bigr)=0
\qquad
(j=1,2,3,4)
$$
for all $n\gg 0$.  Hence $(\Mcal')^{\otimes n}$ is $0$-regular with respect to $L'$
for all $n\gg 0$.

Thus all hypotheses of Theorem~\ref{thm:one-ideal} are satisfied.
\end{proof}

This is the triple to which Theorem~\ref{thm:one-ideal} will be applied in Section~\ref{sec:main}.

\section{Arithmetic evaluation of the volume integral}\label{sec:arithmetic-input}

In this section we evaluate the divisor volume integral arising from the
shifted construction.  We first derive the moving disk formula directly from
the projective bundle volume formula and the circular volume function on the
abelian surface.  For the arithmetic step, however, we do not divide the disk
integral into chambers.  Instead, we reverse the order of integration.  The
radial integral can then be evaluated explicitly, and polar coordinates reduce
the remaining triangle integral to three one-variable integrals of rational
functions after a tangent half angle substitution.  Their universal arctangent
terms cancel exactly, while the logarithmic terms form an explicit algebraic
linear combination whose value is positive.  Baker's theorem then gives
transcendence.

The geometry is controlled by the fixed triangle
$$
T=\operatorname{conv}\{(-2,-2),(2,-2),(4,0)\}.
$$
The critical radii occur when the moving disk is tangent to an edge of $T$ or
passes through a vertex.  They are
$$
\frac{4}{\sqrt{10}},\qquad 2,
\qquad 2\sqrt2,
\qquad 4.
$$
These radii describe the possible changes in the disk geometry; the smallest
one, $4/\sqrt{10}$, identifies the pseudo-effective threshold.  The
noncancellation argument below incorporates all chambers at once, so no
endpoint angle is isolated before the full integral is assembled.

\begin{lemma}\label{lem:weighted-pb-volume}
Let $Y$ be an integral projective variety of dimension $v$, let
$$
E=\bigoplus_{i=0}^r \OO_Y(A_i)
$$
for Cartier divisors $A_0,\ldots,A_r$ on $Y$, and set
$$
X:=\mathbf P_Y(E),
\qquad
\xi:=\OO_X(1),
\qquad
\pi:X\to Y.
$$
We use the quotient convention for projective bundles, so that
$$
\pi_*\OO_X(m\xi)=\Sym^m E
\qquad
(m\geq 0).
$$
Let $d=v+r$.  For every integer $a\geq 1$ and every Cartier divisor $B$ on $Y$, one has
$$
\vol_X(a\xi+\pi^*B)
=
\frac{d!}{v!}
\int_{\Delta_a}
\vol_Y\left(B+\sum_{i=0}^r\mu_iA_i\right)
\,d\mu_1\cdots d\mu_r,
$$
where
$$
\Delta_a
:=
\left\{
(\mu_0,\ldots,\mu_r)\in\RR_{\geq 0}^{r+1}
\ \middle|\ 
\sum_{i=0}^r\mu_i=a
\right\}.
$$
Here the integral is taken in the coordinates $\mu_1,\ldots,\mu_r$, with
$\mu_0=a-\sum_{i=1}^r\mu_i$.
\end{lemma}

\begin{proof}
For every integer $k\geq 1$, the projective-bundle formula and the projection formula give
$$
H^0\left(X,k(a\xi+\pi^*B)\right)
=
H^0\left(Y,\Sym^{ka}(E)\otimes \OO_Y(kB)\right).
$$
Since $E$ is a direct sum of line bundles, its symmetric power decomposes as
$$
\Sym^{ka}(E)
=
\bigoplus_{n_0+\cdots+n_r=ka}
\OO_Y\left(\sum_{i=0}^r n_iA_i\right).
$$
Therefore
$$
h^0\left(X,k(a\xi+\pi^*B)\right)
=
\sum_{n_0+\cdots+n_r=ka}
h^0\left(Y,kB+\sum_{i=0}^r n_iA_i\right).
$$
Write
$$
\mu_i:=\frac{n_i}{k}.
$$
Then $\mu_i\in \frac1k\ZZ_{\geq 0}$ and $\sum_i\mu_i=a$, so the preceding expression becomes
$$
h^0\left(X,k(a\xi+\pi^*B)\right)
=
\sum_{\mu\in \Delta_a\cap (\frac1k\ZZ)^{r+1}}
h^0\left(Y,k\left(B+\sum_{i=0}^r\mu_iA_i\right)\right).
$$

Dividing by $k^d/d!$, where $d=v+r$, gives
$$
\frac{h^0\left(X,k(a\xi+\pi^*B)\right)}{k^d/d!}
=
\frac{d!}{v!}
\cdot
\frac1{k^r}
\sum_{\mu\in \Delta_a\cap (\frac1k\ZZ)^{r+1}}
\frac{
h^0\left(Y,k\left(B+\sum_{i=0}^r\mu_iA_i\right)\right)
}{
k^v/v!
}.
$$

It remains to pass from this lattice sum to the corresponding integral.  Set
$$
K
:=
\left\{
B+\sum_{i=0}^r\mu_iA_i
\ \middle|\ 
\mu\in\Delta_a
\right\}
\subset N^1(Y)_\RR.
$$
This is a compact subset of $N^1(Y)_\RR$.  The uniform estimate cited below is
stated for positive coefficients.  If a lattice point lies on a boundary face
of $\Delta_a$, we omit the zero coordinates and apply the same estimate to the
corresponding subcollection of the divisors $A_i$.  Since $\Delta_a$ has only
finitely many faces, the maximum of the resulting error functions is still
$o(k^v)$.  Moreover, the total coefficient in every summand is
$k+\sum_i n_i=k(1+a)$, a fixed multiple of $k$.  Thus, by the uniform
asymptotic estimate for sections on compact sets of numerical classes, see
\cite[Proposition~3.5.1]{BGGJKM}, there is a function
$\rho_0(N)=o(N^v)$ such that
$$
\left|
h^0(Y,k\Lambda)
-
\frac{k^v}{v!}\vol_Y(\Lambda)
\right|
\leq \rho_0(k(1+a))
$$
uniformly for all rational classes $\Lambda\in K$ that occur from lattice points
$$
\mu\in\Delta_a\cap\left(\frac1k\ZZ\right)^{r+1}.
$$
Since $a$ is fixed,
$$
\frac{\rho_0(k(1+a))}{k^v}
=
(1+a)^v
\frac{\rho_0(k(1+a))}{(k(1+a))^v}
\longrightarrow 0.
$$
Thus, after dividing by $k^v/v!$, the error tends to $0$ uniformly.  Hence
$$
\frac1{k^r}
\sum_{\mu\in \Delta_a\cap (\frac1k\ZZ)^{r+1}}
\frac{
h^0\left(Y,k\left(B+\sum_i\mu_iA_i\right)\right)
}{
k^v/v!
}
=
\frac1{k^r}
\sum_{\mu\in \Delta_a\cap (\frac1k\ZZ)^{r+1}}
\vol_Y\left(B+\sum_i\mu_iA_i\right)
+o(1).
$$

The volume function is continuous on $N^1(Y)_\RR$, and hence its restriction to the compact set
$K$ is continuous and bounded.  Therefore the last lattice sum is the usual Riemann sum for
$$
\int_{\Delta_a}
\vol_Y\left(B+\sum_{i=0}^r\mu_iA_i\right)
\,d\mu_1\cdots d\mu_r.
$$
Taking the limit as $k\to\infty$ gives
$$
\vol_X(a\xi+\pi^*B)
=
\frac{d!}{v!}
\int_{\Delta_a}
\vol_Y\left(B+\sum_{i=0}^r\mu_iA_i\right)
\,d\mu_1\cdots d\mu_r.
$$
This proves the formula.
\end{proof}

\begin{proposition}\label{prop:disk-reduction}
With the notation of Section~\ref{sec:BN}, set
$$
\Theta_t:=tL_0-D_0=tM_s+(t-q)B
\qquad
(\tau_0\leq t\leq q).
$$
Numerically,
$$
\Theta_t\equiv tM_s+\pi^*(\beta(t-q)A).
$$
Then
$$
\vol_{X_s}(\Theta_t)
=
\frac{3N}{2}t^4\Phi(R(t)),
$$
where
$$
R(t):=s+1+\beta-\frac{\beta q}{t},
$$
and
$$
\Phi(R):=
\begin{cases}
\displaystyle
\iint_{T\cap B_R}(R^2-x^2-y^2)\,dx\,dy, & R\geq 0,\\[8pt]
0, & R<0.
\end{cases}
$$
Here $B_R\subset \RR^2$ is the closed disk of radius $R$ centered at the origin, and
$$
T=\operatorname{conv}\{(-2,-2),(2,-2),(4,0)\}.
$$
\end{proposition}

\begin{proof}
Since $D_0=qB$ is a nonzero effective divisor whose numerical class is a positive multiple of the pullback of the ample class $A$, the class $-D_0$ is not pseudo-effective.  Indeed, if $H$ is any ample divisor on $X_s$, then $D_0\cdot H^3>0$, whereas every pseudo-effective divisor has nonnegative intersection with $H^3$.  Hence $\tau_0>0$, so $R(t)$ is defined throughout $[\tau_0,q]$.

Since volume depends only on numerical equivalence, we may compute with the class
$$
\Theta_t\equiv tM_s+\pi^*(\beta(t-q)A).
$$
We first assume that $t\in[\tau_0,q]$ is rational.  Choose $b\geq1$ such that $a:=bt$ is an integer.  Then
$$
b\Theta_t=aM_s+(a-bq)B
\equiv
aM_s+\pi^*(\beta(a-bq)A).
$$
Applying Lemma~\ref{lem:weighted-pb-volume} to
$$
X_s=\mathbf P_Y\left(\bigoplus_{i=0}^2\OO_Y(M_i+sA)\right)
$$
gives the following formula.  Here $\dim Y=2$ and $\dim X_s=4$, so the coefficient in the projective bundle formula is $4!/2!=12$:
$$
\vol_{X_s}(b\Theta_t)
=
12\int_{\Delta_a}
\vol_Y\left(\beta(a-bq)A+\sum_{i=0}^2\mu_i(M_i+sA)\right)
\,d\mu_1d\mu_2.
$$
Set $\mu_i=a\lambda_i=bt\lambda_i$.  Then $\lambda=(\lambda_0,\lambda_1,\lambda_2)$ lies in the standard simplex $\Delta$ and
$$
d\mu_1d\mu_2=b^2t^2d\lambda_1d\lambda_2.
$$
Moreover,
$$
\beta(a-bq)A+\sum_{i=0}^2\mu_i(M_i+sA)
=
b\left(\beta(t-q)A+t\sum_{i=0}^2\lambda_i(M_i+sA)\right).
$$
Since $Y$ is a surface, its volume function is homogeneous of degree $2$.  Therefore
$$
\vol_{X_s}(b\Theta_t)
=
12b^4t^2\int_{\Delta}
\vol_Y\left(\beta(t-q)A+t\sum_{i=0}^2\lambda_i(M_i+sA)\right)
\,d\lambda_1d\lambda_2.
$$
On the other hand, $\dim X_s=4$, so
$$
\vol_{X_s}(b\Theta_t)=b^4\vol_{X_s}(\Theta_t).
$$
Canceling $b^4$ gives
$$
\vol_{X_s}(\Theta_t)
=
12t^2\int_{\Delta}
\vol_Y\left(\beta(t-q)A+t\sum_{i=0}^2\lambda_i(M_i+sA)\right)
\,d\lambda_1d\lambda_2.
$$

Write
$$
M_i=A+v_i,
$$
where
$$
v_0=-2D_1-2D_2,
\qquad
v_1=2D_1-2D_2,
\qquad
v_2=4D_1.
$$
Then
$$
\beta(t-q)A+t\sum_{i=0}^2\lambda_i(M_i+sA)
=
tR(t)A+t\sum_{i=0}^2\lambda_i v_i.
$$
The affine map
$$
\Delta\longrightarrow T,
\qquad
\lambda\longmapsto
\lambda_0(-2,-2)+\lambda_1(2,-2)+\lambda_2(4,0),
$$
is described, in the coordinates $\lambda_0=1-\lambda_1-\lambda_2$, by
$$
x=-2+4\lambda_1+6\lambda_2,
\qquad
y=-2+2\lambda_2.
$$
Its Jacobian is therefore
$$
\left|
\det
\begin{pmatrix}
4&6\\
0&2
\end{pmatrix}
\right|
=8.
$$
Thus, if
$$
(x,y)=\lambda_0(-2,-2)+\lambda_1(2,-2)+\lambda_2(4,0),
$$
then $\sum_i\lambda_i v_i=xD_1+yD_2$ and
$$
\vol_{X_s}(\Theta_t)
=
\frac{3}{2}t^2
\iint_T
\vol_Y\left(tR(t)A+t(xD_1+yD_2)\right)
\,dx\,dy.
$$

Since $t>0$, the circular description of the nef cone in Subsection~\ref{subsec:BN-surface} gives
$$
\vol_Y\left(tR(t)A+t(xD_1+yD_2)\right)
=
\begin{cases}
Nt^2\bigl(R(t)^2-x^2-y^2\bigr),
& R(t)\geq0\text{ and }x^2+y^2\leq R(t)^2,\\[4pt]
0,&\text{otherwise.}
\end{cases}
$$
Substitution gives
$$
\vol_{X_s}(\Theta_t)
=
\frac{3N}{2}t^4
\iint_{T\cap B_{R(t)}}
\bigl(R(t)^2-x^2-y^2\bigr)
\,dx\,dy
=
\frac{3N}{2}t^4\Phi(R(t)).
$$

It remains to remove the rationality assumption on $t$.  The left-hand side is continuous because volume is continuous on $N^1(X_s)_\RR$.  The function $R(t)$ is continuous on $[\tau_0,q]$, and for $R\geq0$ one has
$$
\Phi(R)=\iint_T(R^2-x^2-y^2)_+\,dx\,dy.
$$
On every bounded interval of radii, these integrands are uniformly bounded on the fixed compact triangle $T$ and converge pointwise as $R$ varies.  Dominated convergence therefore shows that $\Phi$ is continuous on $[0,\infty)$.  Moreover, the distance from the origin to $T$ is $4/\sqrt{10}$, so $T\cap B_R=\varnothing$ for $0\leq R<4/\sqrt{10}$.  Hence $\Phi(R)=0$ on this interval, and the definition $\Phi(R)=0$ for $R<0$ makes $\Phi$ continuous on all of $\RR$.  The right-hand side is therefore continuous in $t$, and density of the rational numbers completes the proof.
\end{proof}

\Needspace{22\baselineskip}
\begin{corollary}\label{cor:critical-radii}
Let
$$
T=\operatorname{conv}\{(-2,-2),(2,-2),(4,0)\}.
$$
The distinct radii at which the combinatorial type of $T\cap B_R$ can change are
$$
\frac{4}{\sqrt{10}},
\qquad
2,
\qquad
2\sqrt2,
\qquad
4.
$$
Moreover, $R(t)$ is strictly increasing on $(0,\infty)$.  For every $\rho<s+1+\beta$, the number
$$
t_\rho:=\frac{\beta q}{s+1+\beta-\rho}
$$
is positive and satisfies $R(t_\rho)=\rho$.
\end{corollary}

\begin{proof}
The three edges of $T$ lie on
$$
y=-2,
\qquad
x-y=4,
\qquad
x-3y=4.
$$
The closest points on these edge segments to the origin are, respectively,
$$
(0,-2),
\qquad
(2,-2),
\qquad
\left(\frac25,-\frac65\right).
$$
Their distances from the origin are therefore
$$
2,
\qquad
2\sqrt2,
\qquad
\frac{4}{\sqrt{10}}.
$$
The vertex norms are
$$
\|(-2,-2)\|=\|(2,-2)\|=2\sqrt2,
\qquad
\|(4,0)\|=4.
$$
The combinatorial type of $T\cap B_R$ can change only when $\partial B_R$ is tangent to an edge or passes through a vertex, which gives the stated list.  Finally,
$$
R'(t)=\frac{\beta q}{t^2}>0
\qquad
(t>0),
$$
and solving $R(t)=\rho$ gives the formula for $t_\rho$.
\end{proof}

\begin{figure}[t]
\centering
\begin{tikzpicture}[scale=0.95]

  \def\rA{1.264911}
  \def\rB{2.000000}
  \def\rC{2.828427}
  \def\rD{4.000000}

  \draw[->,thick] (-3.4,0) -- (5.5,0) node[right] {$x$};
  \draw[->,thick] (0,-3.3) -- (0,3.5) node[above] {$y$};

  \draw[densely dashed,gray!65,thin] (0,0) circle[radius=\rA];
  \draw[densely dashed,gray!65,thin] (0,0) circle[radius=\rB];
  \draw[densely dashed,gray!65,thin] (0,0) circle[radius=\rC];
  \draw[densely dashed,gray!65,thin] (0,0) circle[radius=\rD];

  \filldraw[fill=gray!15,draw=black,line width=0.9pt]
    (-2,-2) -- (2,-2) -- (4,0) -- cycle;

  \filldraw[fill=black,draw=black] (-2,-2) circle (2pt);
  \node[below left=1pt] at (-2,-2) {$(-2,-2)$};

  \filldraw[fill=black,draw=black] (2,-2) circle (2pt);
  \node[below right=1pt] at (2,-2) {$(2,-2)$};

  \filldraw[fill=white,draw=white] (4,0) circle (4.2pt);
  \filldraw[fill=black,draw=black] (4,0) circle (2.6pt);
  \node[above right=6pt] at (4,0) {$(4,0)$};

  \filldraw[fill=black,draw=black] (0,0) circle (1.6pt);
  \node[above left=1pt] at (0,0) {$0$};

  \node[text=black,fill=white,inner sep=1.2pt] at (0,-2.42) {$y=-2$};
  \node[text=black,fill=white,inner sep=1.2pt,rotate=45] at (3.00,-1.00) {$x-y=4$};
  \node[text=black,fill=white,inner sep=1.2pt,rotate=18.5] at (0.95,-0.52) {$x-3y=4$};

  \draw[black!65] (132:\rA) -- ++(-0.22,0.34);
  \node[text=black,fill=white,inner sep=1pt] at (-0.92,1.56) {$\frac{4}{\sqrt{10}}$};

  \draw[black!65] (114:\rB) -- ++(-0.18,0.42);
  \node[text=black,fill=white,inner sep=1pt] at (-1.70,2.26) {$2$};

  \draw[black!65] (148:\rC) -- ++(-0.52,0.20);
  \node[text=black,fill=white,inner sep=1pt] at (-2.92,1.92) {$2\sqrt{2}$};

  \draw[black!65] (0:\rD) -- ++(0.48,0.24);
  \node[text=black,fill=white,inner sep=1pt] at (4.55,0.42) {$4$};

\end{tikzpicture}
\caption{The triangle $T=\conv\{(-2,-2),(2,-2),(4,0)\}$ and the four critical radii
$\frac{4}{\sqrt{10}},\,2,\,2\sqrt{2},\,4$.}
\end{figure}
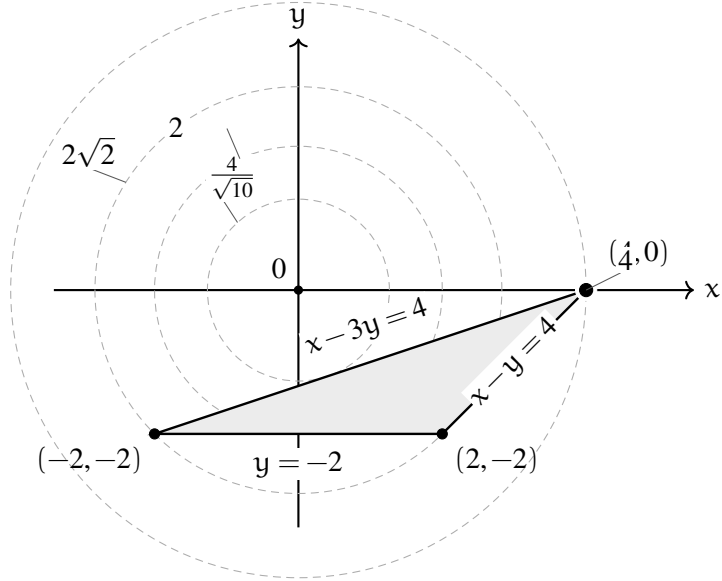

Before proving the arithmetic proposition, we isolate the logarithmic
noncancellation.  The key step is to integrate first in the disk radius.  This
keeps all angular chambers together and makes the cancellation of the
arctangent terms explicit.

\begin{lemma}\label{lem:polar-noncancellation}
Put
$$
a=s+1+\beta,
\qquad
d_0=\frac4{\sqrt{10}},
\qquad
J=\int_{d_0}^{s+1}\frac{\Phi(R)}{(a-R)^6}\,dR.
$$
Then
$$
J=A+\Lambda,
\qquad
A\in\overline{\QQ},
$$
where $\Lambda$ is a finite $\overline{\QQ}$-linear combination of
logarithms of positive algebraic numbers. Moreover,
$$
\Lambda>0.
$$
In particular, $J$ is transcendental.
\end{lemma}

\begin{proof}
Set
$$
U=s+1=a-\beta
$$
and write
$$
r=\sqrt{x^2+y^2}
$$
for $(x,y)\in T$. The distance from the origin to $T$ is $d_0$, and every
point of $T$ has radius at most $4$, whereas $U\geq 6$. Since the integrand is
nonnegative, Tonelli's theorem gives
$$
J
=
\iint_T\int_r^U
\frac{R^2-r^2}{(a-R)^6}\,dR\,dx\,dy.
$$
An antiderivative of the inner integrand is
$$
F_{a,r}(R)
=
\frac{a^2-r^2}{5(a-R)^5}
-
\frac{a}{2(a-R)^4}
+
\frac1{3(a-R)^3}.
$$
At the lower endpoint,
$$
F_{a,r}(r)=\frac{a-4r}{30(a-r)^4}.
$$
At the upper endpoint $R=U$, the quantity $F_{a,r}(U)$ is a rational linear
combination of $1$ and $r^2$. Since $T$ is a rational triangle, the integral
of this upper endpoint term over $T$ is rational. Consequently,
$$
J=A_0-\frac1{30}K,
\qquad
A_0\in\QQ,
$$
where
$$
K:=\iint_T\frac{a-4r}{(a-r)^4}\,dx\,dy.
$$

Set
$$
P_a(r)
=
-\frac{a^2}{(a-r)^3}
+\frac{7a}{2(a-r)^2}
-\frac4{a-r}.
$$
A direct differentiation gives
$$
P_a'(r)=\frac{r(a-4r)}{(a-r)^4}.
$$
The three sides of $T$ lie on
$$
y=-2,
\qquad
x-y=4,
\qquad
x-3y=4.
$$
Solving these equations along a ray $(r\cos\theta,r\sin\theta)$ shows that,
up to boundary sets of measure zero, $T$ is the union of the two polar
regions
$$
\Omega_1:
-\frac{3\pi}{4}\leq\theta\leq-\frac\pi4,
\qquad
L_0(\theta)\leq r\leq L_1(\theta),
$$
and
$$
\Omega_2:
-\frac\pi4\leq\theta\leq0,
\qquad
L_0(\theta)\leq r\leq L_2(\theta),
$$
where
$$
L_0(\theta)=\frac4{\cos\theta-3\sin\theta},
\qquad
L_1(\theta)=\frac{-2}{\sin\theta},
\qquad
L_2(\theta)=\frac4{\cos\theta-\sin\theta}.
$$
Therefore
\begin{align*}
K
&=
\iint_{\Omega_1}\frac{a-4r}{(a-r)^4}\,r\,dr\,d\theta
+
\iint_{\Omega_2}\frac{a-4r}{(a-r)^4}\,r\,dr\,d\theta\\
&=
\int_{-3\pi/4}^{-\pi/4}
\bigl(P_a(L_1(\theta))-P_a(L_0(\theta))\bigr)\,d\theta\\
&\quad+
\int_{-\pi/4}^{0}
\bigl(P_a(L_2(\theta))-P_a(L_0(\theta))\bigr)\,d\theta\\
&=
\int_{-3\pi/4}^{-\pi/4}P_a(L_1(\theta))\,d\theta
+
\int_{-\pi/4}^{0}P_a(L_2(\theta))\,d\theta
-
\int_{-3\pi/4}^{0}P_a(L_0(\theta))\,d\theta.
\end{align*}

We now compute the nonalgebraic parts of these three boundary integrals. The
boundary functions can be written as
$$
L_0(\theta)=d_0\sec(\theta-\theta_0),
\qquad
\theta_0=-\arctan(3),
$$
$$
L_1(\theta)=2\sec\left(\theta+\frac\pi2\right),
\qquad
L_2(\theta)=2\sqrt2\sec\left(\theta+\frac\pi4\right).
$$
For $h>0$, use the tangent half angle substitution
$$
z=\tan\left(\frac\phi2\right),
\qquad
\sec\phi=\frac{1+z^2}{1-z^2},
\qquad
d\phi=\frac{2\,dz}{1+z^2}.
$$
A direct calculation gives
\begin{align*}
&P_a(h\sec\phi)\,d\phi\\
&\qquad=
\left(
-\frac3{a(1+z^2)}
+
\frac{
h^2(1+z^2)\bigl(a(z^2-1)+3h(z^2+1)\bigr)
}{
a\bigl((a+h)z^2-(a-h)\bigr)^3
}
\right)dz.
\end{align*}
Put
$$
\kappa_h=\sqrt{\frac{a-h}{a+h}},
\qquad
\rho_h=\frac{3h^5}{2a(a^2-h^2)^{5/2}}.
$$
Denote the second rational function by $Q_{a,h}(z)$.  Its only poles are
the order-three poles $z=\pm\kappa_h$.  The coefficients of the simple-pole
terms are obtained without choosing any branch of a primitive:
$$
\operatorname{Res}_{z=\kappa_h}Q_{a,h}
=
\frac12
\left.
\frac{d^2}{dz^2}
\bigl((z-\kappa_h)^3Q_{a,h}(z)\bigr)
\right|_{z=\kappa_h}
=
\rho_h,
$$
and, by evenness,
$$
\operatorname{Res}_{z=-\kappa_h}Q_{a,h}=-\rho_h.
$$
Thus
$$
Q_{a,h}(z)
=
\frac{\rho_h}{z-\kappa_h}
-
\frac{\rho_h}{z+\kappa_h}
+
Q^{\rm rat}_{a,h}(z),
$$
where $Q^{\rm rat}_{a,h}$ has only poles of order two or three.  Its primitive
is therefore rational.  At algebraic endpoints it contributes only algebraic
numbers.  The complete list of nonalgebraic primitives is consequently
$$
-\frac3a\arctan z
+
\rho_h\log\left(\frac{\kappa_h-z}{\kappa_h+z}\right).
$$
This identity isolates every logarithmic and arctangent contribution used
below; no omitted part of the partial fraction decomposition can contribute
a further nonalgebraic term.

Set
$$
u=\sqrt2-1,
\qquad
p=\frac{\sqrt5-1}{2},
\qquad
v=\frac{\sqrt{10}-1}{3}.
$$
The corresponding $z$ intervals for $L_1,L_2,L_0$ are, respectively,
$$
[-u,u],
\qquad
[0,u],
\qquad
[-p,v].
$$
Since $a=s+1+\beta\geq8$, one has
$$
\kappa_2\geq\sqrt{\frac35}>\frac12>u,
$$
$$
\kappa_{2\sqrt2}^2
\geq
\frac{8-2\sqrt2}{8+2\sqrt2}
>
\frac14,
\qquad
\text{hence }\kappa_{2\sqrt2}>\frac12>u,
$$
and, since $d_0<4/3$,
$$
\kappa_{d_0}>\sqrt{\frac57}>\frac34>p,v.
$$
In particular, none of the three integration intervals meets a pole.

The arctangent contributions can now be read off without any further
partial fractions:
$$
\begin{array}{c|c|c|c}
\text{boundary} & h & z\text{-interval} &
\text{signed increment of }\arctan z\\ \hline
L_1 & 2 & [-u,u] & 2\arctan(u)=\pi/4\\
L_2 & 2\sqrt2 & [0,u] & \arctan(u)=\pi/8\\
-L_0 & d_0 & [-p,v] &
-\bigl(\arctan(p)+\arctan(v)\bigr)=-3\pi/8
\end{array}
$$
Here $u=\sqrt2-1=\tan(\pi/8)$, so
$$
\arctan(u)=\frac\pi8.
$$
Moreover, $0<p,v<1$ and a direct calculation gives
$$
\frac{p+v}{1-pv}=1+\sqrt2=\tan\left(\frac{3\pi}{8}\right).
$$
The tangent addition formula, together with
$0<\arctan(p)+\arctan(v)<\pi/2$, therefore yields
$$
\arctan(p)+\arctan(v)=\frac{3\pi}{8}.
$$
Therefore the total contribution of the universal arctangent term to $K$ is
$$
-\frac3a
\left(
\frac\pi4+\frac\pi8-\frac{3\pi}{8}
\right)
=0.
$$
The table also makes clear that the cancellation occurs only after all three
boundary integrals are combined.

For $0<z<\kappa_h$, define
$$
\lambda_{h,z}
=
\log\left(\frac{\kappa_h+z}{\kappa_h-z}\right).
$$
The logarithmic primitive associated to the two residue terms is
$$
\rho_h\log\left(\frac{\kappa_h-z}{\kappa_h+z}\right).
$$
Evaluating on the three intervals, taking account of the minus sign in front
of the $L_0$ integral, and absorbing all algebraic endpoint contributions
into a single number $A\in\overline{\QQ}$, we obtain
$$
J=A+\Lambda,
$$
where
$$
\Lambda
=
\frac1{30}
\left(
2\rho_2\lambda_{2,u}
+
\rho_{2\sqrt2}\lambda_{2\sqrt2,u}
-
\rho_{d_0}\bigl(\lambda_{d_0,p}+\lambda_{d_0,v}\bigr)
\right).
$$
The preceding inequalities show that every logarithm is the real logarithm
of a positive algebraic number.

It remains to prove that $\Lambda>0$. Since $\kappa_2<1$ and the function
$$
x\longmapsto\log\left(\frac{x+u}{x-u}\right)
$$
is decreasing for $x>u$, one has
$$
\lambda_{2,u}
>
\log\left(\frac{1+u}{1-u}\right)
=
\log(1+\sqrt2)
>
\frac12.
$$
For the last inequality, one may use
$$
\log(1+\sqrt2)
=
\int_1^{1+\sqrt2}\frac{dt}{t}
>
\frac{\sqrt2}{1+\sqrt2}
=2-\sqrt2
>
\frac12.
$$
Therefore
$$
2\rho_2\lambda_{2,u}>\rho_2.
$$
On the other hand,
$$
\frac{\rho_{d_0}}{\rho_2}
=
\left(\frac{d_0}{2}\right)^5
\left(\frac{a^2-4}{a^2-d_0^2}\right)^{5/2}
<
\left(\frac2{\sqrt{10}}\right)^5
<
\frac18.
$$
The inequalities
$$
p<\frac23,
\qquad
v<\frac34,
\qquad
\kappa_{d_0}>\sqrt{\frac57}>\frac56
$$
give
$$
\frac{p}{\kappa_{d_0}}<\frac45,
\qquad
\frac{v}{\kappa_{d_0}}<\frac9{10}.
$$
Consequently,
$$
\lambda_{d_0,p}<\log9<3,
\qquad
\lambda_{d_0,v}<\log19<3.
$$
It follows that
$$
\rho_{d_0}\bigl(\lambda_{d_0,p}+\lambda_{d_0,v}\bigr)
<
6\rho_{d_0}
<
\frac34\rho_2.
$$
Since $\rho_{2\sqrt2}\lambda_{2\sqrt2,u}>0$, we conclude that
$$
30\Lambda
>
\rho_2-\frac34\rho_2
=
\frac14\rho_2
>0.
$$
Thus $\Lambda>0$, so the logarithmic part is nonzero.
Corollary~\ref{cor:Baker-linear-form} now implies that $J=A+\Lambda$ is
transcendental.
\end{proof}
\begin{proposition}\label{prop:arith}
With the notation of Sections~\ref{sec:BN} and~\ref{sec:arithmetic-input}, the integral
$$
I:=\int_{\tau_0}^{q}\vol_{X_s}(tL_0-D_0)\,dt
$$
is transcendental.
\end{proposition}

\begin{proof}
Set
$$
a=s+1+\beta.
$$
By Proposition~\ref{prop:disk-reduction},
$$
\vol_{X_s}(tL_0-D_0)
=
\frac{3N}{2}t^4\Phi(R(t)),
\qquad
R(t)=a-\frac{\beta q}{t}.
$$
We first identify the lower endpoint.  The pseudo-effective cone is closed, so
$\tau_0 L_0-D_0$ is pseudo-effective.  It is not big: if it were big, then
openness of the big cone would imply that $tL_0-D_0$ is big, hence
pseudo-effective, for some $t<\tau_0$, contradicting the definition of
$\tau_0$.  Therefore
$$
\vol_{X_s}(\tau_0 L_0-D_0)=0.
$$
For every $t>\tau_0$, one has
$$
tL_0-D_0=(\tau_0 L_0-D_0)+(t-\tau_0)L_0,
$$
and the right-hand side is the sum of a pseudo-effective class and an ample
class.  Hence $tL_0-D_0$ is big and has positive volume.  On the other hand,
$tL_0-D_0$ is not pseudo-effective for $t<\tau_0$.  Proposition~\ref{prop:disk-reduction}
therefore shows that the transition from zero to positive volume occurs exactly
when the growing disk first meets the interior of $T$, namely when
$$
R(t)=d_0:=\frac4{\sqrt{10}}.
$$
Since $R(t)$ is continuous and strictly increasing, it follows that
$$
R(\tau_0)=d_0.
$$
Also,
$$
R(q)=s+1.
$$

Changing variables from $t$ to $R$ gives
$$
t=\frac{\beta q}{a-R},
\qquad
t^4\,dt=\frac{(\beta q)^5}{(a-R)^6}\,dR.
$$
Therefore
$$
I
=
\frac{3N}{2}(\beta q)^5
\int_{d_0}^{s+1}\frac{\Phi(R)}{(a-R)^6}\,dR.
$$
The prefactor is a nonzero algebraic number, while the integral is
transcendental by Lemma~\ref{lem:polar-noncancellation}.  Hence $I$ is
transcendental.
\end{proof}

\section{Existence of transcendental epsilon multiplicity}\label{sec:main}

We now assemble the geometric input from Section~\ref{sec:BN}, the arithmetic calculation from Section~\ref{sec:arithmetic-input}, and the one-ideal formula from Theorem~\ref{thm:one-ideal} to prove the main theorem.

\begin{theorem}\label{thm:main}
There exists a $5$-dimensional normal standard graded domain $S$ and a homogeneous ideal $J\subseteq S$ such that
$$
\eps(JS_{S_+})
$$
is transcendental.
\end{theorem}

\begin{proof}
Fix $s,\beta,q,k$ as in Section~\ref{sec:BN}, and let
$$
X:=X_s,
\qquad
S:=R(X,L'),
\qquad
J:=S\cdot H^0(X,qL'-D')\subset S.
$$
By Proposition~\ref{prop:geometric-hypotheses}, $L'$ defines a projectively
normal embedding of the smooth variety $X_s$.  Hence
$S=R(X_s,L')$ is a normal standard graded domain.  Since $\dim X_s=4$, we have
$\dim S=5$, and $J$ is a homogeneous ideal of $S$.

Because $L'=kL_0$ and $D'=kD_0$, the class $tL'-D'$ is pseudo-effective if
and only if $tL_0-D_0$ is pseudo-effective.  Therefore
$$
\tau_{L'}(D')=\tau_{L_0}(D_0)=\tau_0.
$$
Theorem~\ref{thm:one-ideal} applies with $c=q$, and gives
$$
\eps(JS_{S_+})
=
5\int_{\tau_0}^{q}\vol_{X_s}(tL'-D')\,dt.
$$
For every $t\in[\tau_0,q]$,
$$
tL'-D'=k(tL_0-D_0).
$$
Since $\dim X_s=4$, homogeneity of the volume function yields
$$
\vol_{X_s}(tL'-D')
=
k^4\vol_{X_s}(tL_0-D_0).
$$
Consequently,
$$
\eps(JS_{S_+})
=
5k^4\int_{\tau_0}^{q}\vol_{X_s}(tL_0-D_0)\,dt.
$$
The integral is transcendental by Proposition~\ref{prop:arith}.  Multiplication
by the nonzero algebraic number $5k^4$ preserves transcendence, so
$\eps(JS_{S_+})$ is transcendental.
\end{proof}

\section{Further questions on the arithmetic of epsilon multiplicity}\label{sec:future-directions}

Theorem~\ref{thm:main} shows that epsilon multiplicity can have genuinely
transcendental values.  Together with the irrationality examples of
Cutkosky, H\`a, Srinivasan, and Theodorescu \cite{CHST}, this suggests that
epsilon multiplicity should be studied not only as an asymptotic invariant of
ideals, but also as an arithmetic object.  The mechanism of the present paper
points to a more specific question: when can a local cohomological limit be
transported to a divisor-volume integral, and what arithmetic restrictions are
then imposed by the geometry of the resulting volume function?

\begin{question}\label{q:arith-geography}
Which geometric or algebraic hypotheses force $\varepsilon(I)$ to be rational?
More generally, what is the possible arithmetic range of $\varepsilon(I)$ for
ideals of maximal analytic spread that are not $\mathfrak m$-primary?
\end{question}

Theorem~\ref{thm:main} gives transcendental values, while \cite{CHST} gives
irrational values.  On the other hand, monomial, graded, and other structured
settings often exhibit stronger finiteness or computability properties
\cite{JM13,JMV15,DDRV25}.  It would be useful to identify the geometric
features of an ideal, a Rees algebra, or a volume function that separate
rational behavior from algebraic irrationality and transcendence.

\begin{question}\label{q:rigidity}
Can transcendental epsilon multiplicity occur under stronger algebraic
constraints, for instance for ideals in regular local rings, homogeneous
prime ideals, or integrally closed ideals?
\end{question}

The construction in this paper passes through divisorial ideals on section
rings and a projective-bundle volume computation.  It is therefore natural to
ask whether transcendence persists in more rigid classes, or whether such
classes impose arithmetic constraints on epsilon multiplicity.

\begin{question}\label{q:transport}
For which broader classes of ideals or numerical invariants can one express
the leading local cohomological asymptotic directly in terms of divisor
volumes or comparable geometric asymptotics?
\end{question}

A broader transport principle could connect epsilon multiplicity with other
invariants defined by asymptotic length formulas, including invariants in
prime characteristic and birational geometry.  The point would not only be to
produce further transcendence results, but to understand when arithmetic
features of volume functions are preserved by local algebraic constructions.

\begin{center}
Acknowledgments
\end{center}
We would like to thank Iacopo Brivio and Jonathan Montaño for helpful discussions. The second author would also like to thank the Center for Mathematical Sciences and Applications at Harvard and the Hebrew University of Jerusalem for their support during this project. The second author was partially supported by the Israel Science Foundation grant ISF-687/24.


\begin{thebibliography}{99}



\bibitem{BN}
C.~Borntr\"ager and M.~Nickel,
\emph{Algebraic volumes of divisors},
European J. Math. \textbf{5} (2019), no.~4, 1192--1201.

\bibitem{BS}
T.~Bauer and T.~Szemberg,
\emph{On tensor products of ample line bundles on abelian varieties},
Math. Z. \textbf{223} (1996), no.~1, 79--85.



\bibitem{CutSarkar}
S.~D.~Cutkosky and P.~Sarkar,
\emph{Epsilon multiplicity and analytic spread of filtrations},
Illinois J. Math. \textbf{68} (2024), no.~1, 189--210.

\bibitem{Sarkar26}
P.~Sarkar,
\emph{Epsilon multiplicity, multiplicity=volume formula and analytic spread of family of ideals},
preprint, \href{https://arxiv.org/abs/2605.03814}{arXiv:2605.03814} (2026).



\bibitem{Hartshorne}
R.~Hartshorne,
\emph{Ample vector bundles},
Inst. Hautes \`Etudes Sci. Publ. Math. \textbf{29} (1966), 63--94.

\bibitem{Lazarsfeld}
R.~Lazarsfeld,
\emph{Positivity in Algebraic Geometry I},
Ergebnisse der Mathematik und ihrer Grenzgebiete, vol.~48,
Springer, 2004.


\bibitem{Stacks}
\emph{The Stacks Project},
Tags 01Q1, 084M, 01AL, 01O4, 01OB, and 01OC,
\url{https://stacks.math.columbia.edu}.


\bibitem{UlrichValidashti}
B. Ulrich and J. Validashti,
\emph{Numerical criteria for integral dependence},
Math. Proc. Cambridge Philos. Soc. \textbf{151} (2011), no.~1, 95--102.

\bibitem{CHST}
S. D. Cutkosky, H. T\`ai H\`a, H. Srinivasan, and E. Theodorescu,
\emph{Asymptotic behavior of the length of local cohomology},
Canad. J. Math. \textbf{57} (2005), no.~6, 1178--1192.

\bibitem{Cut11}
S. D. Cutkosky,
\emph{Asymptotic growth of saturated powers and epsilon multiplicity},
Math. Res. Lett. \textbf{18} (2011), no.~1, 93--106.

\bibitem{CutLan24}
S. D. Cutkosky and S. Landsittel,
\emph{Epsilon multiplicity is a limit of Amao multiplicities},
J. Algebra Appl. \textbf{24} (2025), nos.~13--14, 2541021.

\bibitem{Lan25}
S. Landsittel,
\emph{Some formulas for epsilon multiplicity in local rings},
Commun. Algebra \textbf{54} (2026), no.~5, 1962--1975.

\bibitem{DRT25}
S. Das, S. Roy, and V. Trivedi,
\emph{Density functions for epsilon multiplicity and families of ideals},
J. Lond. Math. Soc. \textbf{111} (2025), no.~4, e70155.

\bibitem{Fulger}
M. Fulger,
\emph{Local volumes of Cartier divisors over normal algebraic varieties},
Ann. Inst. Fourier (Grenoble) \textbf{63} (2013), no.~5, 1793--1847.

\bibitem{Waldschmidt}
M. Waldschmidt,
\emph{Linear Independence of Logarithms of Algebraic Numbers},
IMSc Report No.~116, Institute of Mathematical Sciences, Madras, 1992.

\bibitem{BGGJKM}
J.~I.~Burgos Gil, W.~Gubler, P.~Jell, K.~K\"unnemann, and F.~Martin,
\emph{Differentiability of non-Archimedean volumes and non-Archimedean Monge--Amp\`ere equations},
with an appendix by R.~Lazarsfeld,
Algebr. Geom. \textbf{7} (2020), no.~2, 113--152.

\bibitem{JM13}
J.~Jeffries and J.~Monta\~no,
\emph{The $j$-multiplicity of monomial ideals},
Math. Res. Lett. \textbf{20} (2013), no.~4, 729--744.

\bibitem{JMV15}
J.~Jeffries, J.~Monta\~no, and M.~Varbaro,
\emph{Multiplicities of classical varieties},
Proc. Lond. Math. Soc. (3) \textbf{110} (2015), no.~4, 1033--1055.

\bibitem{DDRV25}
S.~Das, S.~Dubey, S.~Roy, and J.~K.~Verma,
\emph{Computing epsilon multiplicities in graded algebras},
J. Pure Appl. Algebra \textbf{229} (2025), no.~11, 108107.

\end{thebibliography}
\end{document}